\documentclass[11pt,preprint,authoryear]{elsarticle}

\usepackage{amsmath,amsfonts,amsthm,amssymb,mathtools,float}

\usepackage{lineno}
\newcommand*\patchAmsMathEnvironmentForLineno[1]{%
	\expandafter\let\csname old#1\expandafter\endcsname\csname #1\endcsname
	\expandafter\let\csname oldend#1\expandafter\endcsname\csname end#1\endcsname
	\renewenvironment{#1}%
	{\linenomath\csname old#1\endcsname}%
	{\csname oldend#1\endcsname\endlinenomath}}%
\newcommand*\patchBothAmsMathEnvironmentsForLineno[1]{%
	\patchAmsMathEnvironmentForLineno{#1}%
	\patchAmsMathEnvironmentForLineno{#1*}}%
\AtBeginDocument{%
	\patchBothAmsMathEnvironmentsForLineno{equation}%
	\patchBothAmsMathEnvironmentsForLineno{align}%
	\patchBothAmsMathEnvironmentsForLineno{flalign}%
	\patchBothAmsMathEnvironmentsForLineno{alignat}%
	\patchBothAmsMathEnvironmentsForLineno{gather}%
	\patchBothAmsMathEnvironmentsForLineno{multline}%
}

\usepackage{natbib} \bibpunct{(}{)}{;}{author-year}{}{,}
\newtheorem{theorem}{Theorem}
\newtheorem{corollary}{Corollary}
\newtheorem{proposition}{Proposition}
\newtheorem{remark}{Remark}

\theoremstyle{definition}

\newcommand\E{\mathbb E}
\newcommand\Pb{\mathbb P}
\newcommand\R{\mathbb R}
\newcommand\Dc{\mathcal D}
\newcommand\Xc{\mathcal X}
\newcommand\Gc{\mathcal G}
\newcommand\Hc{\mathcal H}
\newcommand\Fc{\mathcal F}
\newcommand\Vc{\mathcal V}
\newcommand\Ec{\mathcal E}
\newcommand\Lc{\mathcal L}
\newcommand\Rc{\mathcal R}
\newcommand\Pc{\mathcal P}
\newcommand\eff{\mathrm{eff}}

\usepackage{hyperref}
\hypersetup{
	colorlinks,
	citecolor=Blue,
	linkcolor=Red,
	urlcolor=Red}
\usepackage[capitalise]{cleveref}
\usepackage[style=american]{csquotes} 

\usepackage{graphicx}
\usepackage{caption}
\usepackage{subcaption}
\usepackage{bbm}

\newcommand\blfootnote[1]{%
  \begingroup
  \renewcommand\thefootnote{}\footnote{#1}%
  \addtocounter{footnote}{-1}%
  \endgroup
}
\usepackage[hang,flushmargin]{footmisc}

\begin{document}
	
\begin{frontmatter}
	
    \title{A Wright-Fisher graph model and the impact of directional selection on genetic variation}
	
	\author{Ingemar Kaj\(^{a,\ast}\), Carina F.\ Mugal\(^{b,c}\), Rebekka Müller\(^a\)}
\end{frontmatter}
\vspace*{-0.5cm}
\begin{center}
    \today
\end{center}
\vspace{0.5cm}

\section*{Abstract} 

We introduce a multi-allele Wright-Fisher model with non-recurrent, reversible mutation and directional selection. In this setting, the allele frequencies at a single locus track the path of a hybrid jump-diffusion process with state space given by the vertex and edge set of a graph. Vertices represent monomorphic population states and edge-positions mark the biallelic proportions of ancestral and derived alleles during polymorphic segments. 
We derive the stationary distribution in mutation-selection-drift equilibrium and obtain the expected allele frequency spectrum under large population size scaling. 
For the extended model with multiple independent loci we derive rigorous upper bounds for a wide class of associated measures of genetic variation. Within this framework we present mathematically precise arguments to conclude that the presence of directional selection reduces the magnitude of genetic variation, as constrained by the bounds for neutral evolution.

\subsection*{Keywords}

\noindent
Wright-Fisher jump-diffusion process, 
directional selection, 
mutation bias, 
genetic diversity, 
effective mutation rate, 
theoretical population genetics

\blfootnote{
	\(^a\)Department of Mathematics, Uppsala University, Uppsala, Sweden\\ 
	\(^b\)Department of Ecology and Genetics, Uppsala University, Uppsala, Sweden\\
	\(^c\)Laboratory of Biometry and Evolutionary Biology, University of Lyon 1, UMR CNRS 5558, Villeurbanne, France\\
    \(^\ast\)Corresponding author: ikaj@math.uu.se
}

\newpage
\section{Introduction}\label{S:Introduction}

The degree of genetic variation within a population is determined by the interrelations of evolutionary processes such as mutation, genetic drift and natural selection.
Mutation, the fundamental source of genetic variation, is frequently modeled as a non-recurrent mutation mechanism that initializes the segregation of an allele  in the population but otherwise does not influence the population frequency of the allele
\citep{Kimura1969,Sawyer1992,McVean1999}.
Genetic drift and natural selection, on the other hand, control the time span over which mutations segregate in a population until eventually reaching fixation or extinction.
While genetic drift ultimately acts to eliminate genetic variation, different selection mechanisms can either prolong or shorten the time to fixation or extinction. Directional selection, where one of the alleles in a given pair of allelic types has a selective advantage over the other, is commonly viewed as a force to reduce the level of genetic variation. However, as pointed out in \citet{Novak2017}, "rigorous arguments for this idea are scarce".

Population genetics modeling for the purpose of analyzing genetic variation under the combined influence of different evolutionary processes typically builds on some version of Wright-Fisher models \citep{Fisher1930,Wright1931,Wright1938} or Moran type models \citep{Moran1958}. 
As pioneered by \citet{Kimura1964}, diffusion approximation techniques under scaling of evolutionary time and large population size are instrumental and helped advance the understanding of the distribution of inherited allele frequencies, both dynamically and under steady-state, see e.g.\ \cite{Durrett2008,Etheridge2011}. 
The original Wright-Fisher model with mutation and selection \citep{Wright1931} concerns the allele frequency distribution for two allelic types and distinguishes recurrent and non-recurrent mutation.  
With regards to approximating allele frequencies with trajectories of diffusion processes, non-recurrent mutation occur on the boundary of the state space and provide the renewal of polymorphic segments. Recurrent mutation is ongoing and appears through linear drift terms in the diffusion generator \citep{Etheridge2011}. 

In this work we apply diffusion approximation methods to study a multi-allele and multi-locus model with non-recurrent, reversible mutation and directional selection in an isolated population assuming independence among loci. 
It is a consequence of the assumption of the non-recurrent mutation mechanism that the overall mutational input is small enough to prevent additional allelic types at a locus that is already polymorphic. This is consistent with observations in empirical data where multi-allelic single nucleotide variation is typically rare \citep{Cao2015,Phillips2015}. 
In addition, mutation is reversible since we work with a fixed, finite number of allelic types and all mutation events involving a given pair of alleles may take place in both directions. Within this framework our objective is essentially to show that the presence of directional selection reduces the magnitude of genetic variation, as constrained by the bounds for neutral evolution. To this aim we derive stationary distributions over monomorphic and polymorphic states in mutation-selection-drift equilibrium. Closed form expressions for the expected allele frequency spectrum are obtained asymptotically under large population size scaling and rigorous upper bounds are derived for a wide class of associated measures of genetic variation under the influence of directional selection. 

To put our approach in context, the extension from studying the evolutionary dynamics of genetic loci with two types to general multi-allele Wright-Fisher models with a fixed number $K\ge 2$ possible allelic states for each genetic locus, can be traced back to \citet{Wright1949}. For the case of recurrent mutation mechanisms such $K$-allele models have been developed in much detail. 
The state space for single locus frequencies is now (a subset of) the $K$-simplex, which presents considerable challenges in extracting useful probabilistic information. For a brief history of K-allele Wright-Fisher models with recurrent mutation, we refer to \cite{Ferguson2018}, and for some of the mathematical results to \citet[][Ch. 4 and 5]{Etheridge2011}.  
A recent approximation approach to multi-allele models with recurrent mutation \citep{Burden2016,Ferguson2018} starts from the presumption that mutation events are rare on the time scale of evolution relevant for the diffusion approximation. Then, with sufficiently small mutation rates, the allele frequencies will be mostly concentrated either on the vertices of the $K$-simplex, or on the edges connecting a pair of mutating alleles and with only a small fraction of probability mass remaining on simplex domains that allow three or more alleles existing simultaneously. 

In contrast, our approach towards modeling the multi-allelic case relies on non-recurrent and reversible mutation in between a fixed number of genetic types.
A jump-diffusion process is introduced, biallelic by construction, with state space consisting only of the vertices and edges of the graph subset of the simplex. Some key features of the jump-diffusion process are already implemented in \cite{Mugal2014} and \cite{Kaj2016} for the simpler setting of arbitrary ancestral-and-derived alleles. In the graph model, the vertices correspond to the presence of a specific fixed type (or allele) in a genetic locus and the edges between two types represent continuous polymorphic states. Similar boundary mutation multi-allele models have been discussed in the context of synonymous codon usage \citep{Zeng2010} or so-called polymorphism-aware phylogenetic models \citep{Maio2013,Borges2019}, with a focus on methodological development for statistical inference from genomic data.

\section{A Wright-Fisher graph model} 
A continuous time Markov process with state space given by a connected, directed graph $\Gc=(\Vc,\Ec)$ with vertices $\Vc$ and unit length edges $\Ec$ captures the random change of allelic types at a single locus. Such a locus can be of abstract nature (mutant versus wild type) or specific, for instance consisting of $l$ nucleotides in the genome, where $l=1$ corresponds to a single site on the genome or $l=3$ to nucleotide triplets, such as protein-coding codons. 
The finite vertex set represents the available allelic types at the locus, while the family of edges allows for keeping track of the possible mutations among types and any polymorphic state. Each edge is a directed interval of length one, starting in a vertex $u\in\Vc$ and leading to another vertex $v\in\Vc $, such that the position on the edge records the relative frequency of type $v$ as a mutant derived from ancestral type $u$. 
The relevant graph-valued model is a hybrid jump and diffusion process with compact state space, in which the open edges form a continuous interior and the vertices are discrete boundary points. Mutation events occur only on the boundary. Each mutation is succeeded by a polymorphic phase of two alleles co-existing in the population, upon which the Wright-Fisher diffusion determines the frequency and subsequent extinction and fixation probabilities of the mutant. The graph in \cref{fig:graph} illustrates an example state space on which the graph-valued process moves.

Formally, we consider a Markov process $X=(X_t)_{t\ge 0}$, with values in the compact state space $D$ formed by the graph equipped with intervals $[0,1)$ associated with each edge in $\Ec$, directed from $0$ to $1$, and encoded by a triplet $X_t=(U_t,V_t,Y_t)\in \Vc\times \Vc\times [0,1)$. 
The boundary of the state space consists of the point set $\partial D=\{(u,u,0)\colon u\in\Vc\}$, where the boundary state $(u,u,0)$ represents a monomorphic locus at which the entire population has the same allelic type $u\in \Vc$. 
A state $(u,v,y)$ within the interior of the state space
\begin{equation*}
    D^\circ=\{(u,v,y)\colon u\in \Vc,  v\in \Vc,  u\not= v, 0<y<1\}
\end{equation*}
lies on the directed edge leading from vertex state $(u,u,0)$ to vertex state $(v,v,0)$.
It arises when a mutation from $u$ to $v$ occurred and brought mutants of type $v$ to be present in the population at relative frequency in the infinitesimal interval $(y,y+dy)$. 
Consequently, the interior state $(v,u,y)$, located on the complementary edge directed in the opposite direction from $v$ to $u$, assigns relative frequency $y$ to mutant type $u$ derived from an ancestral $v$. 
Finally, the closure $D=D^\circ\cup \partial D$ is reached along the limits
\begin{equation*}
    (u,v,y)\to
  \left\{
  \begin{array}{cc}
    (u,u,0)\in \partial D,& y\to 0\\
    (v,v,0)\in \partial D,& y\to 1 
  \end{array}
  \right. ,
  \quad (u,v,y)\in D^\circ.
\end{equation*}

\begin{figure}[!ht]
    \centering
    \includegraphics[width=\textwidth]{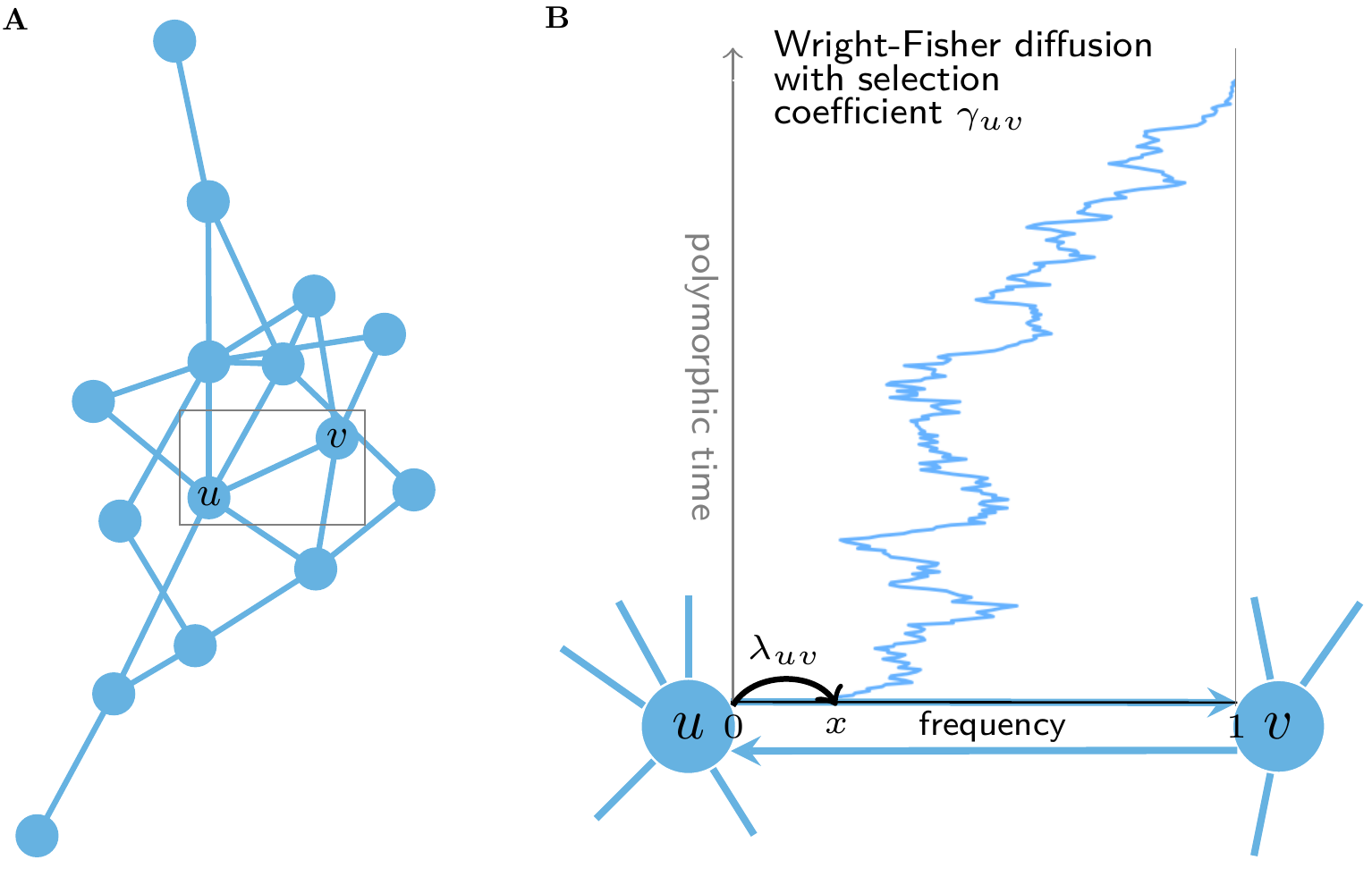}
    \caption{Panel A: A connected graph as state space for the graph-valued process. Between each pair of connected vertices there are in fact two directed edges. Panel B: Zoom-in of directed edges between types $u$ and $v$, equipped with unit intervals. A mutation on the boundary from type $u$ to type $v$ occurs with intensity $\lambda_{uv}$. The initial frequency of the mutant $v$ is a fixed value $x\in(0,1)$ which is also the initial value for a Wright-Fisher diffusion with selection coefficient $\gamma_{uv}$ that is started by the mutation. The process either goes to fixation in type $v$ or to extinction in $u$.}
    \label{fig:graph}
\end{figure}

\subsection{Reversible boundary mutation}

The mutation mechanism of the process $X$ is specified by a fixed entry point $x\in(0,1)$ and a family of nonnegative mutation rate parameters $\{\lambda_{uv}\colon v\not= u\}$. Here, $x$ represents the fraction of a population that is affected by a single mutation. 
We thus define $\lambda_{uv}/x$ as the population mutation intensity from $u$ to $v$ per evolutionary time unit, where the evolutionary time unit corresponds to "$x^{-1}$ generations". This means $\lambda_{uv}$ can be thought of as the population mutation rate "per generation", which commonly represents the macroscopic mutation rate. The motivation for this particular choice becomes clear in the subsequent \cref{sec:population_scaling} when we introduce a parameter $N$ for population size, take $x=1/N$, and run mutations at a rate of speed $\lambda_{uv}/x=N\lambda_{uv}$.

The graph edge set consists of those edges $e_{uv}$ between types for which the mutation intensity is positive, $\Ec=\{e_{uv}\colon u,v\in \Vc,\, \lambda_{uv}>0\}$. 
We postulate that every type is essential, that all mutations are reversible, and that the graph $\Gc$ is irreducible, by assuming that the mutation rates satisfy the conditions
\begin{itemize}
\item[i)]  $\lambda_u:=\sum_{v:v\not=u} \lambda_{uv}>0$ for each $u\in \Vc$,
\item[ii)]  $\lambda_{uv}>0 \iff \lambda_{vu}>0$,
\item[iii)]  for each pair of allelic types, $u,v\in \Vc$, there is a sequence of edge mutations $u\to w_1\to \dots \to w_n\to v$, such that 
 $\lambda_{u,w_1}\cdot \ldots \cdot \lambda_{w_n,v}>0$.
\end{itemize}
The first assumption guarantees that no type is a mutational trap, as the total intensity $\lambda_u$ of a mutation from $u$ to some other type is strictly positive. The second condition ensures that any mutation from one type to another may also occur in the reverse direction, i.e.\ mutation is reversible.
The third condition entails the assumption that every type may be reached from every other type by a chain of non-recurrent mutation events all occurring on the boundary.

The hybrid jump and diffusion mechanism of $X$ is such that, starting in a boundary point $z=(u,u,0)\in \partial D$, the process holds during an exponential time with intensity $\lambda_u/x$. It then jumps to an interior point $z'=(u,v,x)\in D^\circ$ governed by jump rates $\lambda_{uv}$, which represents mutant type $v$ entering the population at (continuous) fraction $x$. 
Assuming that a jump from $(u,u,0)$ to $(u,v,x)$ occurs at time $r$, the interior trajectory of $X$ is a continuous path
\begin{equation*}
    X_t=(u,v,Y_t^r),\quad r\le t<\tau,
\end{equation*}
where $Y^r_t$, $t\ge r$, is a diffusion in $D$ starting from $Y^r_r=x$ such that the path $\xi_s^r=Y^r_{r+s}$, $s\ge 0$, is a Wright-Fisher process with selection, which solves
\begin{equation}\label{eq:WFdiffusion}
    d\xi_s=\gamma_{uv}\,\xi_s(1-\xi_s)\,ds + \sqrt{\xi_s(1-\xi_s)}\,dB^{uv}_s,  \quad \xi_0=x.
\end{equation}
We denote by $\tau$ the first exit time of the interior state space $D^\circ$, i.e. $X_\tau=(u,u,0)$ if $\xi_s$ is absorbed in 0 and $X_\tau=(v,v,0)$ if absorbed in 1. The former case is extinction and the latter case is fixation of the allelic type $v$. 
For edges $e_{uv}\in \Ec$ the parameter $\gamma_{uv}$ denotes the selection coefficient for mutations from $u$ to $v$. 
Moreover, $\{B^{uv}, e_{uv}\in \Ec\}$ are i.i.d.\ standard Brownian motions. 
\Cref{fig:graph}B depicts the hybrid jump and diffusion setup.

\subsection{Directional selection}

To express genic selection in the model we let the family of selection coefficients $\gamma_{uv}$ assigned to the edges $e_{uv}$ in the graph satisfy the anti-symmetric condition
\begin{equation}\label{def:timereversible_sel}
    \gamma_{uv}=-\gamma_{vu},\quad e_{uv}\in \Ec.
\end{equation}
The central instance is directional selection based on a static fitness landscape, where each allelic type is assigned a (time-independent) fitness level $F_u$, $u\in \Vc$, and each polymorphic pair of alleles $(u,v)$ has relative selection coefficient
\begin{equation}\label{def:fitness_sel}
    \gamma_{uv}=F_v-F_u,\quad e_{uv}\in \Ec,
\end{equation}
i.e. $\gamma_{uv}=-\gamma_{vu}$.  
Consequently, in each pair the type with the highest fitness has the selective advantage at the expense and equal disadvantage of the other type. 
Another relevant example of anti-symmetric selection coefficients is the preferential fixation of strong (C and G) over weak (A and T) nucleotides due to the process of GC-biased gene conversion (gBGC) \citep{Duret2009,Mugal2015}, which analytically is equivalent to directional selection \citep{Nagylaki1983}.

We note that under assumption (\ref{def:timereversible_sel}) the distribution of the process $(X_t)$ simplifies on each pair of edges $e_{uv}$ and $e_{vu}$ through the equality in distribution
\begin{equation*}
    (u,v,Y_t^r)\stackrel{d}{=} (v,u,1-Y_t^r), \quad r\le t<\tau.
\end{equation*}

\subsection{Properties of the Wright-Fisher graph process} 

\paragraph{Green's function}

The graph process $(X_t)_{t\ge 0}$ restricted to a particular edge $e_{uv}\in \Ec$, is a classical Wright-Fisher diffusion, described in \cref{eq:WFdiffusion}, with selection coefficient $\gamma_{uv}$.
We write $\Pb^\gamma_x$ for the probability measure and $\E^\gamma_x$ for the expectation of the process starting at $x$, and select the associated scale function $S_\gamma(x)$ and speed function $m_\gamma(x)$ as
\begin{equation*}
    S_0(x)=x,\quad  S_\gamma(x)=\frac{1}{2\gamma}(1-e^{-2\gamma x}),\quad \gamma\not=0,
    \qquad m_\gamma(x)=\frac{e^{2\gamma x}}{x(1-x)}.
\end{equation*}
Since $m_\gamma$ is not integrable near $0$ or $1$, both points $\{0,1\}$ are exit boundary points, therefore accessible from the interior of the state space. The diffusion can reach either of these boundaries but will stay at the point hit first. The time $\tau_0$ required to reach $0$ is the extinction time, the time $\tau_1$ to reach $1$ the fixation time, and $\tau=\min(\tau_0, \tau_1)$ is the exit time of the interior interval $(0,1)$. The corresponding fixation probability 
$q_\gamma(x)$ equals $q_\gamma(x)=S_\gamma(x)/S_\gamma(1)$
\citep{Kimura1962}, hence
\begin{equation*}
    q_\gamma(x)=\Pb^\gamma_x(\tau_1<\tau_0)=\frac{1-e^{-2\gamma x}}{1-e^{-2\gamma}}, \quad \gamma\not=0,
    \qquad q_0(x)=x.
\end{equation*}
The occupation time functional
\begin{equation}\label{eq:green_integral}
    \E^\gamma_x\Big[\int_0^{\tau}g(\xi_s)\,ds\Big] =\int_0^1 G_\gamma(x,y)g(y)\,dy,
\end{equation}
is obtained from the Green function $G_\gamma(x,y)$ defined by
\begin{equation}\label{eq:green_def}
    G_\gamma(x,y)=\left\{
    \begin{array}{ll} 
    2q_\gamma(x)(S_\gamma(1)-S_\gamma(y))m_\gamma(y), & 0\le x\le y\le 1, \\[2mm]
    2(1-q_\gamma(x))(S_\gamma(y)-S_\gamma(0))m_\gamma(y), & 0\le y\le x\le 1.
    \end{array}
    \right. 
\end{equation}
Assuming we start from $z=(u,v,x)\in D^\circ$, the possible transitions from $z$ to $z'$ before exiting $D^\circ$ are those such that $z'=(u,v,y)$ for some $y$, $0<y<1$. For such a pair, Green's function is $G(z,z')=G_{\gamma_{uv}}(x,y)$ and governs the transition of the process from $z$ towards $z'$. 
For background on mathematical population genetics and more detailed  properties of the Wright-Fisher diffusion process with selection, we refer the reader to e.g.\ \citet{Maruyama1977,Karlin1981,Ewens2004,Etheridge2011}.

\paragraph{Invariant measure on the boundary} 

By replacing the polymorphic excursions of $X$ with instantaneous jumps, we obtain an embedded continuous time Markov chain. Indeed, starting in $u\in\Vc$ the embedded chain holds during an exponential time with rate $\lambda_u/x$, then with probability $\lambda_{uv}/\lambda_u$ picks type $v$ and with probability $q_{\gamma_{uv}}(x)$ jumps to the new type $v$. Taken together, the transition rate of the embedded chain from $u$ to $v$ is $h_{uv}/x$, where $h_{uv}=\lambda_{uv}q_{\gamma_{uv}}(x)$.  
The invariant measure on the boundary is a probability measure $\eta^x=\{\eta^x(z)=\eta_u^x,z=(u,u,0)\in \partial D\}$, which, since the state space is finite, is a steady-state for the embedded chain. 
For a pair of vertices $u$, $v$, the invariant distribution must satisfy
\[
\eta_u^x\lambda_{uv}q_{\gamma_{uv}}(x) 
=\eta_v^x \lambda_{vu} (1-q_{\gamma_{uv}}(1-x)),
\]
since fixation of a mutant $v$ starting from frequency $x$ is the same as extinction of the ancestral type $u$ starting from frequency $1-x$.
It is straightforward to verify that assumption (\ref{def:timereversible_sel}) for the selection coefficients in the model now yields the 
relationship
\begin{equation}\label{def:detailed_balance}
    \eta_u^x h_{uv} = \eta_u^x \lambda_{uv} q_{\gamma_{uv}}(x)
    =\eta_v^x \lambda_{vu} q_{\gamma_{vu}}(x)
    =\eta_v^x h_{vu}, \quad e_{uv}\in \Ec.
\end{equation}
This shows that the embedded chain satisfies a detailed balance equation across each pair of edges of the graph. The "current of probability" from $u$ to $v$, given by $\eta_u^x h_{uv}$, equals the corresponding flow $\eta_v^x h_{vu}$ from $v$ to $u$. 

Hence there exists a unique invariant and reversible measure $(\eta_u^x)_{u\in\Vc}$ on the boundary, typically associated with time-reversibility. The matrix $\Hc_x=(h_{uv}(x))/x$ with diagonal elements $h_{uu}(x)=-\sum_{v} \lambda_{uv} q_{\gamma_{uv}}(x)$ is the infinitesimal generator of the continuous time Markov chain $\widehat U_t=U_t|Y_t=0$, for which the conditioning formalizes the notion of instantaneous jumps and hence suppressed polymorphisms. 

We are now in the position to connect the probability weights on the boundary given by $\eta^x$ with the occupation measure in the interior of the state space as provided by the Green function. 
If the process starts from the boundary according to the reversible measure $\eta^x$, i.e.\ the initial distribution of $X_0$ is $\eta^x$, then the dynamics of the first jump is governed by
\begin{equation}\label{def:nu}
  \nu^x = \sum_{u,v\in\Vc} \eta_u^x \lambda_{uv}\delta_{(u,v,x)}.
\end{equation} 
In words, the intensity of the first jump is determined by the accumulation of jump intensities over all vertices and its outgoing edges, weighted by the probability to be in a particular vertex. The subsequent relative position on a particular edge is then given by the fixed entry point $x$.
Under $\nu^x$, the relevant Green function contribution to $z'=(u,v,y)$ is    
\begin{equation}\label{eq:green_nu}
    G(\nu^x,z') = \eta_u^x\lambda_{uv} G_{\gamma_{uv}}(x,y),
\end{equation}
which also justifies writing 
\begin{equation}\label{def:G_dz'}
    G(\nu^x,dz') = \eta_u^x\lambda_{uv} G_{\gamma_{uv}}(x,y)\,dy. 
\end{equation}

\paragraph{The generator}

We consider real-valued functions $f$ defined on $D$, writing $f(z)=f_{uu}(0)$ for $z=(u,u,0)\in\partial D$ and  
$f(z)=f_{uv}(y)$ for $z=(u,v,y)\in D^\circ$, and let $f_{uv}'(y)$ and $f_{uv}''(y)$ denote first and second order derivatives with respect to $y$ defined in the interior $D^\circ$ of $D$.
The infinitesimal generator of the Markov process $(X_t)_{t\ge 0}$ is the operator $\Lc$ which acts on a suitable domain $\Dc$ of functions $f\colon D\to \R$, twice continuously differentiable in the interior $D^\circ$, by 
\begin{equation*}
    \Lc f(z)=\sum_{v\in\Vc} \lambda_{uv} (f_{uv}(x)-f_{uu}(0)), \quad  z=(u,u,0)\in\partial D, 
\end{equation*}
and 
\begin{equation*}
    \Lc f(z)=
    \gamma_{uv} y(1-y)\,f_{uv}'(y) + \frac{1}{2}y(1-y) \, f_{uv}''(y), 
    \quad z=(u,v,y)\in D^\circ.
\end{equation*}
Let $\Fc$ denote the class of real-valued bounded functions on $D$ such that $f\in\Dc$, and such that $f_{uv}(y)\to f_{uu}(0)$ as $y\to 0$ and $f_{uv}(y)\to f_{vv}(0)$ as $y\to 1$, $u,v\in\Vc$.
 
\subsection{Stationary distribution}

In order to determine the equilibrium behavior of the Wright-Fisher graph process we will construct a stationary distribution, i.e.\ a distribution which is preserved under the time-dynamics of the model and hence represents the typical probability weight assigned to the various parts of the graph in steady-state. It can be shown in addition that the graph process satisfies exponential ergodicity and that the stationary distribution is the unique limit distribution. In this extended sense, the stationary distribution measures how likely it is that the process visits a certain position on the graph after allowing sufficient time to reach an equilibrium. However, the proof of exponential ergodicity is extensive and outside the scope of the work at hand. 

For each edge $e_{uv}\in\Ec$ let $D_{e_{uv}}=\{(u,v,y)\colon u,v\in\Vc, 0\leq y<1\}$ be the subset of $D$ which consists of the monomorphic state $(u,u,0)$ and all polymorphic states $(u,v,y), 0\leq y<1,$ between $u$ and $v$. Then $\cup_{e_{uv}\in \Ec} D_{e_{uv}}=D$ and the intersection of two edge sets contains any shared vertex.
A measure $\mu$ on $D$ is stationary for $(X_t)_{t\ge 0}$, by definition, if the balance equations
\begin{equation}\label{eq:balance_eq_generator}
    \int_D  \Lc f(z)\,\mu(dz)  
    = \sum_{e_{uv}\in\Ec}\int_{D_{e_{uv}}}  \Lc f(z)\,\mu(dz)= 0,\quad f\in \Fc,
\end{equation}
hold. 
We say that the measure $\mu$ is edge-reversible for the Wright-Fisher graph process $(X_t)_{t\ge 0}$, if the detailed balance edge equations
\[
\int_{D_{e_{uv}}} \Lc f(z)\,\mu(dz) + \int_{D_{e_{vu}}} \Lc f(z)\,\mu(dz) = 0,\quad f\in \Fc,
\]
hold for every pair of edges $e_{uv},e_{vu}\in \Ec$. 
By summing this relation over all pairs $u$ and $v$, linked by the two edges $e_{uv}$ and $e_{vu}$, we recover  
\cref{eq:balance_eq_generator}. Thus, an edge-reversible measure $\mu$ yields a stationary distribution of the Wright-Fisher graph process.

\begin{theorem}\label{thm:stationary}
There exists an edge-reversible measure $\mu$ for $(X_t)_{t\ge 0}$ on $D$, which is given by
\begin{align*}
  \mu(z)  &= \frac{\eta^x(z)}{1+ \int_{D^\circ} G(\nu^x,dz')}
  ,\quad z\in\partial D,\\
  \mu(dz) &=\frac{G(\nu^x,dz)}{1+\int_{D^\circ} G(\nu^x,dz')},\quad z\in D^\circ,
\end{align*}
where $\eta^x(z)$, $z\in\partial D$, is the unique boundary measure defined by the detailed balance equations in \cref{def:detailed_balance}, $\nu^x$ is be the averaged jump measure in \cref{def:nu}, both dependent on $x$, and $G(\nu^x,dz')$ is the measure on $D^\circ$ introduced in \cref{def:G_dz'}. 
\end{theorem}

\begin{proof}   
Put $\Omega=1+\int_{D^0} G(\nu^x,dz')$. 
To verify that $\mu$ is edge-reversible we need to establish for each pair of edges $e_{uv},e_{uv}\in\Ec$ the identity
\begin{equation}\label{eq:generator_stat}
\begin{split}
    &\Omega\Big( \int_{D_{e_{uv}}} \Lc f(z)\,\mu(dz) + \int_{D_{e_{vu}}} \Lc f(z)\,\mu(dz) \Big)\\
    &=\eta_u^x\lambda_{uv}(f_{uv}(x)-f_{uu}(0)) 
     +\eta_u^x\lambda_{uv}\int_0^1 \Lc f_{uv}(y)\,G_{\gamma_{uv}}(x,dy)\\
    &\quad +\eta_v^x\lambda_{vu}(f_{vu}(x)-f_{vv}(0))
     +\eta_v^x\lambda_{vu}\int_0^1 \Lc f_{vu}(y)\,G_{\gamma_{vu}}(x,dy) =0,
\end{split}
\end{equation}
$f\in\Fc$. 
Here, using \cref{eq:green_def},
\begin{equation*}
    \int_0^1 \Lc f_{uv}(y)\,G_{\gamma_{uv}}(x,dy) = q_{\gamma_{uv}}(x) A_{uv}(x)+ (1-q_{\gamma_{uv}}(x)) B_{uv}(x)
\end{equation*}
with
\begin{equation*}
    A_{uv}(x)
    = \int_x^1 
    \Big(f_{uv}'(y) + \frac{1}{2\gamma_{uv}}f_{uv}''(y)\Big)
   (1-e^{-2\gamma_{uv}(1-y)}) \,dy
\end{equation*}
and
\begin{equation*}
    B_{uv}(x)
    =\int_0^x 
    \Big(f_{uv}'(y) + \frac{1}{2\gamma_{uv}}f_{uv}''(y)\Big)
    (e^{2\gamma_{uv}y}-1)\,dy.
\end{equation*}
Partial integration twice in each of $A_{uv}(x)$ and $B_{uv}(x)$ yield, noticing that $f_{uv}(y)\to f_{vv}(0)$ as $y\to 1$, 
\begin{equation*}
    A_{uv}(x)=f_{vv}(0)-f_{uv}(x) - \frac{1}{2\gamma_{uv}} f_{uv}'(x) (1-e^{-2\gamma_{uv}(1-x)})
\end{equation*}
and
\begin{equation*}
    B_{uv}(x)=
    f_{uu}(0)-f_{uv}(x) + \frac{1}{2\gamma_{uv}} f_{uv}'(x) (e^{2\gamma_{uv}x}-1). 
\end{equation*}
Since
\[
q_{\gamma_{uv}}(x)\frac{1-e^{-2\gamma_{uv}(1-x)}}{2\gamma_{uv}}
    - (1-q_{\gamma_{uv}}(x)) \frac{e^{2\gamma_{uv}x}-1}{2\gamma_{uv}}=0,
\]
it follows that    
\begin{align*}
&q_{\gamma_{uv}}(x) A_{uv}(x)
     +(1-q_{\gamma_{uv}}(x))B_{uv}(x)\\
&\quad=q_{\gamma_{uv}}(x)(f_{vv}(0)-f_{uv}(x))
+(1-q_{\gamma_{uv}}(x)) (f_{uu}(0)-f_{uv}(x)),
\end{align*}
which is
\begin{align*}
\int_0^1 \Lc f_{uv}(y)\,G_{\gamma_{uv}}(x,dy)
=-(f_{uv}(x)-f_{uu}(0))
+q_{\gamma_{uv}}(x)(f_{vv}(0)-f_{uu}(0)).
\end{align*}
Similarly, by symmetry,
\begin{align*}
\int_0^1 \Lc f_{vu}(y)\,G_{\gamma_{vu}}(x,dy)
=-(f_{vu}(x)-f_{vv}(0))
+q_{\gamma_{vu}}(x)(f_{uu}(0)-f_{vv}(0)).
\end{align*}
Thus, by combining \cref{eq:generator_stat} with the detailed balance equation (\ref{def:detailed_balance}) for the boundary measure $\eta^x$, 
\begin{align*}
    &\Omega\Big(\int_{D_{e_{uv}}} \Lc f(z)\,\mu(dz) + \int_{D_{e_{vu}}} \Lc f(z)\,\mu(dz)\Big)\\
    &=
    \eta_u^x \lambda_{uv} q_{\gamma_{uv}}(x)(f_{vv}(0)-f_{uu}(0))
    + \eta_v^x \lambda_{vu} q_{\gamma_{vu}}(x)(f_{uu}(0)-f_{vv}(0))=0,
\end{align*}
and therefore $\int_D \Lc f(z)\,\mu(dz) = 0$ in view of \cref{eq:balance_eq_generator}.
\end{proof} 

\begin{remark}
\citet{Peng2013} study diffusion processes defined on an open, bounded domain $D^0$ in $\R^d$ with holding and jumping from a regular boundary $\partial D$, and provide existence and uniqueness of a stationary distribution under suitable assumptions on the regularity of the coefficients of the generator of the diffusion process. Our result for the case of the graph-valued process is parallel to \citet[][Theorem 4.2]{Peng2013}.
\end{remark}

\section{Large population size scaling}\label{sec:population_scaling}

The polymorphic segments of the path of $(X_t)_{t\geq0}$ through the interior $D^\circ$ run on the time scale of evolution, which is a characteristic of the Wright-Fisher diffusion. The generic re-scaling approach behind the Wright-Fisher diffusion approximation considers the change in frequency in a population of size $N$ over the time span of $N$ generations. Simultaneously, the relevant selection coefficient at the level of generations, $s$, is of the order $s\sim \gamma/N\to 0$, where $\gamma$ is the selection coefficient of the limiting diffusion process. 
To properly adapt the holding time distribution in the present model to the evolutionary time scale we therefore introduce a parameter $N$ as a proxy of population size and prescribe that the jumps into the interior of the state space have size $x=1/N$. The time scale of the system is then set by the speed of mutation $\lambda_{uv}/x= \lambda_{uv} N$.
Our goal in this section is to analyze the stationary distribution $\mu$ in \cref{thm:stationary} with $x=1/N$ for large but fixed $N$. Specifically we identify the dominant terms in the asymptotic expansion of $\mu$ under scaling for large $N$ and drop remainder terms of order $\ln N/N$ and smaller. During this procedure it is convenient to make a number of simplifying approximations valid formally in the limit $N\to\infty$. 
It is important to keep in mind however that the population size proxy $N$ is kept as a finite model parameter. 

\subsection{Approximation of the stationary distribution}

We recall that in our model two vertices $u$ and $v$ are always connected by two directed edges $e_{uv}$ and $e_{vu}$ whenever the jump rates between $u$ and $v$ are positive. 
For each edge $e_{uv}$, as $N\to \infty$, we introduce the scaled fixation probability $\omega_\gamma$, where $\gamma=\gamma_{uv}$, by
\begin{equation}\label{def:omega}
  q_{\gamma}(1/N) =
  \frac{\omega_{\gamma}}{N}+O\Big(\frac{1}{N^2}\Big),\quad
  \omega_\gamma=\frac{2\gamma}{1-e^{-2\gamma}},\quad \gamma\not=0,\qquad
  \omega_0=1.
\end{equation}
Due to assumption (\ref{def:timereversible_sel}) on directional selection, we obtain the symmetry relation
\begin{equation}\label{eq:symmetry_omega}
    \omega_{\gamma_{vu}}=\omega_{-\gamma_{uv}}= e^{-2\gamma_{uv}}\, \omega_{\gamma_{uv}}.
\end{equation}
As before, the collection of jump rates $\{\lambda_{uv}\}$ and selection coefficients $\{\gamma_{uv}\}$ associated with the vertices $\Vc$ and edges $\Ec$ again define an embedded scaled continuous time Markov chain on $\Vc$.  
The scaled generator matrix arises as $N \Hc_{1/N}\sim \Hc$, with
\begin{equation*}
    \Hc=(h_{uv}),\quad h_{uv}=\lambda_{uv} \omega_{\gamma_{uv}},\quad -h_{uu}=\sum_{v\in\Vc} \lambda_{uv} \omega_{\gamma_{uv}},
\end{equation*}
and is irreducible. In analogy with the previous relation
(\ref{def:detailed_balance}), the unique solution $\eta=\{\eta(z),z\in \partial D\}$ of the detailed balance equations
\begin{equation}\label{def:detailed_balanceN}
    \eta_v \lambda_{vu}\omega_{\gamma_{vu}}=\eta_u \lambda_{uv}\omega_{\gamma_{uv}},\quad e_{uv}\in \Ec,
\end{equation}
is the scaled invariant boundary distribution of the embedded Markov chain with generator $\Hc$. 
The solution $\eta$ of (\ref{def:detailed_balanceN}), that no longer depends on $N$, is a convenient approximation of the solution $\eta^x$ of (\ref{def:detailed_balance}) with $x=1/N$.  The distribution of the first jump averaged over the scaled invariant measure,
\begin{equation}\label{def:nu_scaled}
\nu^{1/N}=\sum_{u,v\in \Vc} \eta_u\lambda_{uv}\delta_{(u,v,1/N)},
\end{equation}
still depends on the initial mutation frequency $1/N$. 
The next result records the dominant terms in \cref{thm:stationary}, where we have fixed all mutation and selection parameters, and then choose $x=1/N$ and $N$ large enough so that remainder terms of order $O(\ln N/N)$ and smaller are removed.

\begin{proposition}\label{prop:statapprox}
The stationary single site distribution $\mu$ in \cref{thm:stationary} satisfies for large $N$ the approximation
\begin{equation*}
\mu(z)=\mu_N(z)+O(1/N),\quad z\in D,
\end{equation*}
where the approximating distribution $\mu_N$ has monomorphic site probabilities 
\begin{equation*}
    \mu_N(u,u,0)=\frac{\eta_u}{\Omega'_N},\quad u\in\Vc,
\end{equation*}
for $z=(u,u,0)$, 
polymorphic density given by
\begin{equation*}
\begin{split}
    \mu_N(u,v,y)\,dy 
    =\frac{2\eta_u \lambda_{uv}}{\Omega'_N}
    \Big\{ &\frac{\omega_{\gamma_{uv}}(1-e^{-2\gamma_{uv}(1-y)})}
      {2\gamma_{uv} y(1-y)}\,1_{\{1/N<y<1\}} \\
    &\quad
      + (N-\omega_{\gamma_{uv}})\,1_{\{0<y<1/N\}} \Big\}\,dy,
\end{split}
\end{equation*}
for $z=(u,v,y)$, and is normalized by $\Omega'_N=\Omega_N+O(\ln N/N)$, with
\begin{equation}\label{eq:OmegaN}
    \Omega_N = 1 + 2 \sum_{u,v\in\Vc} \eta_u \lambda_{uv} (1+\ln N+K_{\gamma_{uv}}) 
\end{equation}
and
\begin{equation*}
    K_\gamma=\omega_\gamma \int_0^1 (-\ln y)  (e^{-2\gamma y}-e^{-2\gamma(1-y)})\,dy.
\end{equation*}
\end{proposition}
Before proving \cref{prop:statapprox}, we comment on some properties of the function $K_\gamma$ and state the approximate distribution $\mu_N$ for the case of neutral evolution.
\begin{remark}\label{rem:K}
First, we have $K_\gamma\ge 0$, $\gamma\ge 0$. Second, the function $K_\gamma$ is odd, $K_{-\gamma}=-K_\gamma$. In particular, $K_{\gamma_{vu}}=K_{-\gamma_{uv}}= -K_{\gamma_{uv}}$. Third, $K_\gamma$ grows logarithmically for large $\gamma$: with $\gamma_e=0.5772\dots$
denoting Euler's constant,
\begin{equation*}
    K_{\gamma}\le  \left\{
    \begin{array}{ll}
        \gamma, & 0\le \gamma\le 1\\
        \gamma_e+\ln 2\gamma,  &\gamma\ge 1, 
    \end{array} \right.
    \qquad  K_{\gamma}\sim  \gamma_e+\ln 2\gamma\quad   \mbox{for large $\gamma$}.
\end{equation*}
\end{remark}

\begin{remark}\label{rm:neutral_evolution}
For the special case of neutral evolution, $\gamma_{uv}=0$ for all $e_{uv}\in\Ec$, we have $\omega_0=1$ and $K_0=0$, so
\begin{equation*}
    \mu_N(u,u,0)=\frac{\eta_u}{\Omega_N'},\quad u\in\Vc,
\end{equation*}
and
\begin{equation*}
    \mu_N(u,v,y)\,dy =\frac{2\eta_u \lambda_{uv}}{\Omega_N'}
    \Big\{ y^{-1}\,1_{\{1/N<y<1\}} + (N-1)\,1_{\{0<y<1/N\}} \Big\}\,dy,
\end{equation*}
for $(u,v,y)\in D^\circ$, where $(\eta_u)$ is the unique solution of the balance equations
\begin{equation*}
    \eta_v \lambda_{vu}=\eta_u \lambda_{uv},\quad e_{uv}\in \Ec,
\end{equation*}
and $\Omega_N'=\Omega_N+O(1/N)$ the normalization factor under neutrality with
\begin{equation*}
    \Omega_N = 1 + 2 (1+\ln N) \sum_{u\in\Vc} \eta_u \lambda_u.
\end{equation*}
\end{remark}

\begin{proof}
For fixed $z'=(u,v,y)$, 
\begin{equation*}
 G(\nu^{1/N},z') =  \eta_u \lambda_{uv} NG_{\gamma_{uv}}(1/N,y). 
\end{equation*}
Here, using \cref{def:omega} with large $N$ and $\gamma=\gamma_{uv}$, 
\begin{align*}
    NG_{\gamma}\Big(\frac{1}{N},y\Big)  
    &= Nq_{\gamma}\Big(\frac{1}{N}\Big) \frac{1-e^{-2\gamma(1-y)}}{\gamma y(1-y)}\,1_{\{1/N<y<1\}} \\
    &\qquad+N\Big(1-q_{\gamma}\Big(\frac{1}{N}\Big)\Big) \frac{e^{2\gamma y}-1}{\gamma  y(1-y)}\,1_{\{0<y<1/N\}}\\
    &= \omega_{\gamma} \frac{1-e^{-2\gamma(1-y)}}{\gamma y(1-y)}\,1_{\{1/N<y<1\}} \\
    &\qquad+ 2(N-\omega_\gamma)\,1_{\{0<y<1/N\}} +O\Big(\frac{1}{N}\Big),
\end{align*}
from which we obtain $\Omega'_N \mu_N(z)$, $z\in D^\circ$.
Moreover, 
\begin{equation*}
\begin{split}
    &N \int_0^1 G_\gamma\Big(\frac{1}{N},y\Big)\,dy
    =\omega_\gamma \int_{1/N}^1\frac{1-e^{-2\gamma(1-y)}}{\gamma y(1-y)}\,dy + 2 + O\Big(\frac{1}{N}\Big).  
\end{split}
\end{equation*}
By partial integration the remaining integral evaluates to 
\begin{align*}
\int_{1/N}^1\{y+&(1-y)\}\frac{1-e^{-2\gamma(1-y)}}{\gamma y(1-y)}\,dy\\
  &= \int_0^{1-1/N} \frac{1-e^{-2\gamma y}}{\gamma y}\,dy 
  + \int_{1/N}^1 \frac{1-e^{-2\gamma(1-y)}}{\gamma y}\,dy\\
  &= \frac{2\ln N}{\omega_\gamma} + 2\int_0^1 (-\ln y) (e^{-2\gamma y}-e^{-2\gamma(1-y)})\,dy + O\Big(\frac{\ln N}{N}\Big).
\end{align*}
Hence, 
\begin{equation*}
    N \int_0^1 G_\gamma\Big(\frac{1}{N},y\Big)\,dy = 2 (1+\ln N+K_\gamma) +O\Big(\frac{\ln N}{N}\Big).
\end{equation*}
Integration over $D^\circ$ yields
\begin{equation*}
    \int_{D^0} G(\nu^{1/N},z')\,dz' =
    2 \sum_{u,v\in\Vc} \eta_u \lambda_{uv} (1+\ln N+K_{\gamma_{uv}}) 
    +O\Big(\frac{\ln N}{N}\Big).
\end{equation*}
The representation of an approximate distribution $\mu_N(z)$ as stated now follows from \cref{thm:stationary}.

Finally, to verify the claims in \cref{rem:K}, for $\gamma>0$, 
\begin{equation*}
\begin{split}
    \frac{K_\gamma}{\omega_\gamma} 
    &= \int_0^{1/2} (-\ln y)  (e^{-2\gamma y}-e^{-2\gamma(1-y)})\,dy\\
    & \quad -\int_0^{1/2} (-\ln (1-y))  (e^{-2\gamma y}-e^{-2\gamma(1-y)})\,dy
   \ge 0.
\end{split}
\end{equation*}
The relation $\omega_{-\gamma}=e^{-\gamma} \omega_\gamma$ implies the symmetry $K_{-\gamma}=-K_\gamma$. 
Furthermore,  the change-of-variable $x=2\gamma y$ yields 
\begin{equation*}
    K_\gamma
    = \frac{\omega_\gamma}{2\gamma} \int_0^{2\gamma} (-\ln x + \ln(2\gamma)) e^{-x}\,dx
    -\frac{\omega_\gamma}{2\gamma} \int_0^{2\gamma} -\ln(1- x/2\gamma)\, e^{-x}\,dx.
\end{equation*}
The rightmost integral is positive and tends to zero as $\gamma\to\infty$,
by an application of the monotone convergence theorem.
Also, $\omega_\gamma/(2\gamma)\to 1$ as $\gamma\to\infty$. Thus,
\begin{equation*}  
   K_\gamma - \ln(2\gamma)\sim \frac{\omega_\gamma}{2\gamma} \int_0^{2\gamma} (-\ln x) e^{-x}\,dx \to  \gamma_e,\quad \gamma\to\infty.
\end{equation*}
\Cref{rm:neutral_evolution} follows directly from \cref{prop:statapprox} for $\gamma=0$.
\end{proof}

\subsection{Allele frequency spectra}\label{sec:stat_distr_undirected_edges}

The unfolded allele frequency spectrum (AFS) describes the allele frequency distribution of the derived allele at a biallelic locus. The AFS is a summary statistic of the stationary distribution and can be retrieved using the modeling setup with two directed edges between each pair of types and the large $N$ approximation in \cref{prop:statapprox}. For a given locus, the unfolded AFS representing the density of the derived allele frequency of any type corresponds to $\sum_{u,v\in\Vc}\mu_N(u,v,y), 0<y<1$.
We visualize the polymorphic density on two directed edges $e_{uv}$ and $e_{vu}$ of a multi-allele model for $\gamma_{uv}=1$ in \cref{fig:mu_N}A. 
\begin{figure}[!ht]
    \centering
    \includegraphics[width=\textwidth]{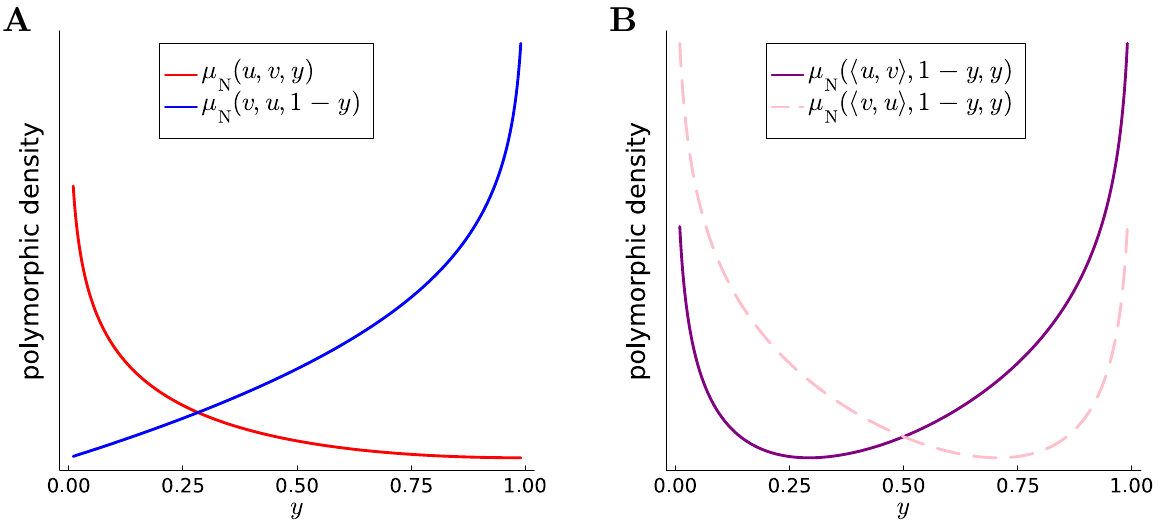}
    \caption{Polymorphic densities of two types $u$ and $v$ with selection coefficient $\gamma_{uv}=1$. 
    Panel A: Relative log-scaled densities for type $u$ at frequency $1-y$ on the directed edge $e_{uv}$, $\mu_N(u,v,y)$ in red, and on the directed edge $e_{vu}$, $\mu_N(v,u,1-y)$ in blue.
    Panel B: The relative log-scaled density on the edges of the pair $\langle u,v \rangle$ for type $u$ at frequency $1-y$, $\mu_N(\langle u,v \rangle,1-y,y)$ in purple, which is the sum of the two curves in panel A. The relative log-scaled density for type $v$ at frequency $1-y$, $\mu_N(\langle v,u \rangle,1-y,y)$ in pink, is symmetric to the purple curve around $y=0.5$.}
    \label{fig:mu_N}
\end{figure}

In practice, the unfolded AFS relies on a polarization of polymorphisms into derived and ancestral types, knowledge that requires additional information such as outgroup data, which is not always readily available. If this is the case a folded AFS can be derived from data. The folded AFS takes biallelic observations and typically measures the minor allele at some frequency $y \in [0,0.5]$, and the other allelic type at complementary frequency $1-y \in [0.5,1]$. 
To formalize representations of unfolded and folded allele frequency spectra using the stationary distribution in the current model, we introduce the set of unordered pairs of vertices $\Pc\coloneqq\{\langle u,v \rangle\colon u,v\in\Vc\}$, with $|\Pc| =\binom{|\Vc|}{2}$. Restricting to the edges of the pair $\langle u,v \rangle$, the polymorphic density of the process when type $v$ has frequency $y$ and type $u$ frequency $1-y$ equals
\[
\mu_N(\langle u,v \rangle, 1-y, y) \coloneqq \mu_N(u,v,y)+\mu_N(v,u,1-y).
\]
The folded density of the minor allele on $\langle u,v \rangle$ is therefore
\[
\mu_\mathrm{fold}(\langle u,v \rangle, y)\coloneqq
\mu_N(\langle u,v \rangle, 1-y, y) + 
\mu_N(\langle v,u \rangle, 1-y, y),\quad 0<y\le 1/2.  
\]
\Cref{fig:mu_N}B depicts the polymorphic density $\mu_N(\langle u,v \rangle, 1-y, y)$ on edges connecting a pair $\langle u,v \rangle$.

\begin{corollary}\label{cor:edgesum}
We have
\begin{equation*}
\begin{split}
    \mu_N(\langle u,v \rangle, 1-y, y)\,dy &=
    \frac{2\eta_u\lambda_{uv}}{\Omega_N'} 
    \Big\{\frac{e^{2\gamma_{uv}y}}{y(1-y)}\,1_{\{1/N<y<1-1/N\}} \\
    &\qquad + N\,1_{\{0<y<1/N\}} + Ne^{2\gamma_{uv}} \,1_{\{1-1/N<y<1\}} \Big\}\,dy,
\end{split}
\end{equation*}
and 
\begin{equation*}
\begin{split}
    \mu_\mathrm{fold}(\langle u,v \rangle, y)\,dy =
    \frac{2\eta_u\lambda_{uv}}{\Omega_N'} 
    \Big\{&\frac{e^{2\gamma_{uv}y}+e^{2\gamma_{uv}(1-y)}}{y(1-y)}\,1_{\{1/N<y\le 1/2\}} \\
    &+ N(1+e^{2\gamma_{uv}}) \,1_{\{0<y<1/N\}} \Big\}\,dy,
\end{split}
\end{equation*}
with the normalization factor $\Omega_N'$ in \cref{prop:statapprox}. 
\end{corollary}

\begin{proof} 
Consider a fixed pair of types $\langle u,v \rangle \in\Pc$. The density of the process when $v$ has frequency $y$ follows from adding up the two densities on each directed edge derived in \cref{prop:statapprox},
\begin{equation*}
\begin{split}
   &\Omega_N'(\mu_N(u,v,y)+\mu_N(v,u,1-y))\\
   &= 2\eta_u\lambda_{uv}\left\{ \frac{\omega_{\gamma_{uv}}(1-e^{-2\gamma_{uv}(1-y)})}{2\gamma_{uv}y(1-y)} 1_{\{1/N<y<1\}}
   + (N-\omega_{\gamma_{uv}}) 1_{\{0<y<1/N\}}  \right\} \, dy\\
   &\, + 2\eta_v\lambda_{vu} \left\{ \frac{\omega_{\gamma_{vu}}(1-e^{-2\gamma_{vu}y})}{2\gamma_{vu}y(1-y)} 1_{\{0<y<1-1/N\}}
   + (N-\omega_{\gamma_{vu}}) 1_{\{1-1/N<y<1\}} \right\} \, dy.
\end{split}
\end{equation*}
Using detailed balance, \cref{def:detailed_balanceN}, and the relationship $\gamma_{vu}=-\gamma_{uv}$, the contribution from the interior, $1/N<y<1-1/N$, for large $N$ is
\begin{equation*}
\begin{split}
    &\eta_u \lambda_{uv} \omega_{\gamma_{uv}} \left\{ \frac{1-e^{-2\gamma_{uv}(1-y)}}{\gamma_{uv}y(1-y)} + \frac{1-e^{-2\gamma_{vu}y}}{\gamma_{vu}y(1-y)} \right\} 
    = 2 \eta_u \lambda_{uv} \frac{e^{2\gamma_{uv}y}}{y(1-y)}. 
\end{split}
\end{equation*}
For large $N$, the contribution from close to the boundary at zero, $0<y<1/N$, is
\begin{equation*}
    2\eta_u \lambda_{uv} (N-\omega_{\gamma_{uv}}) 
    + \eta_v\lambda_{vu} \frac{\omega_{\gamma_{vu}}(1-e^{-2\gamma_{vu}y})}{\gamma_{vu}y(1-y)}
    = 2\eta_u \lambda_{uv} N + O\Big(\frac{1}{N}\Big),
\end{equation*}
and from close to the boundary at one, $1-1/N<y<1$, 
\begin{equation*}
\begin{split}
    &\eta_u \lambda_{uv} \frac{\omega_{\gamma_{uv}}(1-e^{-2\gamma_{uv}(1-y)})}{\gamma_{uv}y(1-y)} 
    + 2\eta_v\lambda_{vu} (N-\omega_{\gamma_{vu}})\\
    &= 2\eta_u\lambda_{uv} N e^{2\gamma_{uv}} + O\Big(\frac{1}{N}\Big),
\end{split}
\end{equation*}
where we again used detailed balance and relation \cref{eq:symmetry_omega}.
As the alternative view of merging the directed edges simply entails reshuffling contributions, the normalization constant does not change.
The expression for $\mu_\mathrm{fold}(\langle u,v \rangle, y)\,dy$ follows from similar calculations.
\end{proof}

An example of the unfolded AFS in a four type model with state space shown in \cref{fig:AFSs}A is given as the orange curve in \cref{fig:AFSs}B.
Similarly as for the unfolded AFS, the sum $\sum_{\langle u,v \rangle\in\Pc}\mu_\textrm{fold}(\langle u,v \rangle,y), 0<y\le 1/2,$ yields the folded, type-independent distribution of derived allele frequencies at a locus (\cref{fig:AFSs}B, red curve). 

\begin{figure}[!ht]
    \centering
    \includegraphics[width=\textwidth]{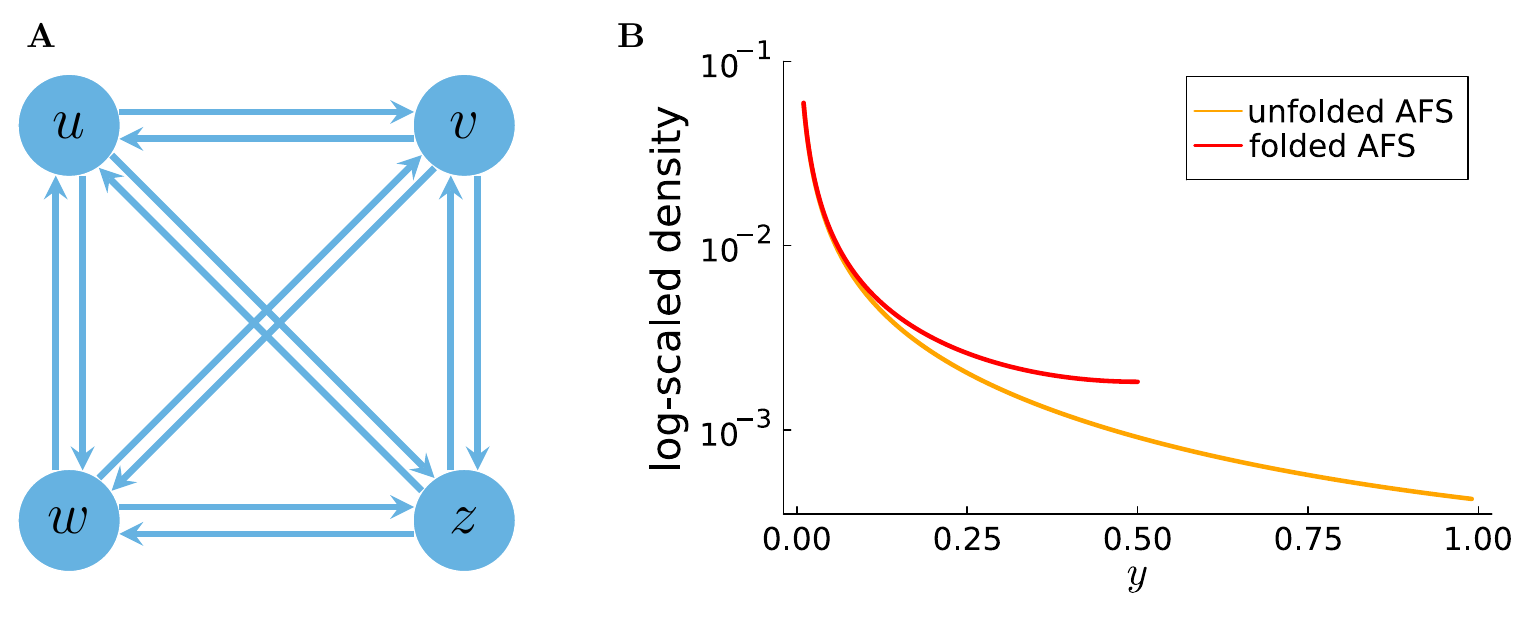}
    \caption{Allele frequency spectra for a four type model.
    Panel A: State space for a model with four types. 
    Panel B: Log-scaled unfolded (blue) and folded (red) AFS.
    Parameters: population size $N=10^4$, selection coefficients defined as fitness differences with $F_u=F_z=0$ and $F_v=F_w=1$, and equal mutation intensity $\lambda=N \times 10^{-8}$ among all types.}
    \label{fig:AFSs}
\end{figure}

\subsection{Extension to multiple loci}\label{sec:multiple_loci}

The model introduced here applies directly to a collection of $L$ loci which evolve independently. For the special case that a locus represents a single site, a collection of consecutive sites represents a DNA sequence. Thus, even though the following considerations are in general about a collection of independent loci, we may use the term \enquote{sequence} instead.
The state of the sequence is defined by a collection of independent holding and jumping diffusion processes $X^j$, $j=1,\dots,L$, with values in the direct product set $\prod_{j=1}^L \Gc^j$, where all graphs $\Gc^j$ have the same vertex set $\Vc$ and edge set $\Ec$.
We allow the set of selection coefficients $\gamma^j_{uv}$, $1\le j\le L$, to differ from one locus to another, but assume that the transition rates $\lambda_{uv}$ are the same among loci or along the sequence. 
For this it is convenient to introduce scaled intensities $\theta_{uv}=\lambda_{uv} L$, $u,v\in\Vc$. The results in \cref{prop:statapprox,cor:edgesum} then apply with $\lambda_{uv}=\theta_{uv}/L$. Summing over $v$, the total intensity $\theta_u=\sum_{v\in\Vc}\theta_{uv}$ of a mutation from type $u$ becomes $\theta_u=\lambda_u L$, $u\in\Vc$, and so $N\theta_u$ is the total rate in the collection of loci per time unit of a mutation affecting $u$. In contrast, the steady-state probabilities $(\mu^j(z))_{z\in D}$ and the boundary probabilities $(\eta^j_u)_{u\in\Vc}$ typically vary between loci, $1\le j\le L$. The total mutation rate on the boundary, $\widehat\theta$, averaged across loci, is
\begin{equation}\label{eq:thetahat}
    \widehat\theta=\frac{1}{L}\sum_{j=1}^L \widehat\theta^j,
\quad   \widehat\theta^j  =\sum_{u\in\Vc} \eta_u^j\theta_u.
\end{equation}
The quantities $\widehat\theta^j$ are the jump rates of the distribution $\nu^{1/N}$ in \cref{def:nu_scaled}.
Clearly, 
\begin{equation}\label{eq:thetamax}
\theta_\mathrm{min}:=\min_{u\in\Vc} \theta_u
\le \widehat\theta 
\le \max_{u\in\Vc}\theta_u =:\theta_\mathrm{max}. 
\end{equation}
The closely related summation
\begin{equation}\label{eq:thetahateff}
    \widehat\theta_\eff=
\sum_{u\in\Vc} \widehat \mu_u\theta_u,
\quad \widehat \mu_u=\frac{1}{L}\sum_{j=1}^L
\frac{\eta^j_u}{\Omega_N^j}
\end{equation}
represents the \emph{effective mutation rate} of the sequence, weighted by the average probability that loci are monomorphic. Of course, $\widehat\theta_\eff\le \widehat\theta$. We say that the mutation mechanism on the graph is \emph{homogeneous} if the total rates in each vertex coincide, i.e.\ $\theta_u=\theta$, $u\in\Vc$. 
Under the stronger assumption of homogeneous mutation, 
\begin{equation}\label{eq:thetahomogene}
    \widehat \theta =\theta,
    \qquad \widehat\theta_\eff=\frac{1}{L}\sum_{j=1}^L \frac{1}{\Omega_N^j}\, \theta.
\end{equation}

We summarize the steady state of the collection of loci by considering the measure-valued process $\Xc_\infty=\sum_{j=1}^L \delta_{X_\infty^j}$, where $X^j_\infty$ has the stationary single locus distribution $\mu^j$ on $\Gc^j$.
The corresponding unfolded AFS across multiple loci is given by
\begin{equation*}
    \frac{1}{L}\sum_{j=1}^L\sum_{u,v\in\Vc}\mu_N^j(u,v,y),\quad 0<y<1.
\end{equation*}
For suitable functions $f$, $f(z)=f_{uv}(y)$, $z\in D$, 
\begin{equation*}
    \langle\Xc_\infty,f\rangle=\sum_{j=1}^L f(X_\infty^j),\quad f\in \Fc,
\end{equation*}
represents the sequence equilibrium distribution. The steady-state expectation under the approximate large population size distribution $\mu_N$ in \cref{prop:statapprox} is \begin{equation}\label{eq:mean_steadystate}
\begin{split}
    \E_N\langle\Xc_\infty,f\rangle
    =\sum_{j=1}^L \Big\{
    &\sum_{u\in\Vc} f_{uu}(0)\mu^j_N(u,u,0) \\
    &\qquad +\sum_{u,v\in \Vc}\int_0^1 f_{uv}(y)\,\mu^j_N(u,v,y)\, dy\Big\}.
\end{split}
\end{equation}

\section{Impact of directional selection on genetic variation} 

There exists a number of summary statistics to assess genetic variation in a population or population sample. These arise as the result of evaluating functionals $\E_N\langle\Xc_\infty,f\rangle$ for specifically chosen functions $f$ and can be analyzed by using \cref{eq:mean_steadystate}.
The first term in the sum over $L$ in (\ref{eq:mean_steadystate}) provides the weight of the boundary probabilities $(\mu^j(u,u,0))_{1\le j\le L}$ over monomorphic loci, and the second term adds the relevant contributions from the allele frequency spectrum of the polymorphic loci. 
We begin with a list of the basic instances of such statistics. As a reference for each case  we specialize to neutral evolution and derive the relevant neutral summary statistics. Under the assumption  $\gamma_{uv}=0$ for every $u,v\in\Vc$, \cref{eq:mean_steadystate} simplifies into
\begin{equation}\label{eq:mean_steadystate_neutral}
\begin{split}
    &\E_N\langle\Xc_\infty,f\rangle
    =\frac{L}{\Omega_N} \sum_{u\in\Vc} f_{uu}(0)\eta_u \\ 
    &\quad +\frac{2}{\Omega_N}\sum_{u,v\in \Vc}\eta_u\theta_{uv}
    \Big(f_{uv}(0+)+ \int_{1/N}^1 \frac{f_{uv}(y)}{y}\,dy
    \Big) +\Rc_N
\end{split}
\end{equation}
with
\[
\Omega_N=1+\frac{2(1+\ln N)\widehat\theta^{\,0}}{L},\quad 
\widehat\theta^{\,0}=\sum_{u\in\Vc} \eta_u\theta_u\le \theta_\mathrm{max},
\]
and $\Rc_N=O(\ln N/N)$. Here, $(\eta_u)_{u\in \Vc}$ is the solution of the neutral detailed balance equation, that is, $\eta_u\theta_{uv}=\eta_v\theta_{vu}$, for every $u,v\in \Vc$, c.f.\ \cref{rm:neutral_evolution}. 
Under homogeneous mutation, 
$\widehat\theta^{\,0}=\theta$.

\subsection{Summary statistics under neutral evolution}\label{lst:summary_statistics_neutral} 

The following listing is derived from (\ref{eq:mean_steadystate_neutral}) with the remainder term $\Rc_N$ suppressed. The additional approximation assuming $\ln N/L$ is not too large, indicated by writing $\sim$ instead of $=$, falls within the approximation range of $\Rc_N$. 
\begin{itemize}
\item[i)] {\bfseries Average effective mutation rate}

Define $f$ by $f_{uu}(0)=\theta_u$ and $f_{uv}(y)=0$. The expected value $\widehat\theta_\eff^{\,0}=\E_N\langle\Xc_\infty,f\rangle/L$ is the average effective mutation rate per sequence under neutral evolution, specifically taking into effect that mutations only occur on the boundary of the graph. The first term in (\ref{eq:mean_steadystate_neutral}) yields
\[ 
\widehat\theta_\eff^{\,0}
= \frac{{\widehat\theta}^{\,0}}{1+2 L^{-1}(1+\ln N)\widehat\theta^{\,0}}
\sim {\widehat\theta}^{\,0}\Big(1- \frac{2(1+\ln N)}{L}\widehat\theta^{\,0}\Big).
\]

\item[ii)] {\bfseries Polymorphic allele functionals}

Using the effective mutation rate in i), we observe for functions $f$ that act on edges only and are independent of the type, i.e.\ with $f_{uu}(0)=0$ and $f_{uv}(y)=f(y)$, $u,v\in\Vc$, that (\ref{eq:mean_steadystate_neutral}) has the form
\begin{equation*}
\begin{split}
    \E_N\langle\Xc_\infty,f\rangle
    &= 2 \widehat\theta_\eff^{\,0} 
    \Big(f(0+)+ \int_{1/N}^1 \frac{f(y)}{y}\,dy
    \Big)\\
    & \leq 2 \widehat\theta^{\,0} 
    \Big(f(0+)+ \int_{1/N}^1 \frac{f(y)}{y}\,dy
    \Big).
\end{split}
\end{equation*}

\item[iii)] {\bfseries Number of monomorphic sites}

Let $f=f^\partial$ be the indicator function on the boundary $\partial D$. The expected number of monomorphic sites out of $L$ is 
\[
\E_N\langle\Xc_\infty,f^\partial\rangle=\frac{L}{\Omega_N} 
=\frac{L}{1+2 L^{-1}(1+\ln N)\widehat\theta^{\,0}}
\sim L-2(1+\ln N)\,\widehat\theta^{\,0}.
\]

\item[iv)] {\bfseries Number of polymorphic sites}

Let $f^\circ=1-f^\partial$.
For $L$ sufficiently large compared to $\ln N$ we obtain the familiar approximation of the expected number of polymorphic sites as 
\[
\E_N\langle\Xc_\infty,f^\circ\rangle=
L-\frac{L}{\Omega_N}
= 2(1+\ln N)   \widehat\theta_\eff^{\,0}
\le 2(1+\ln N)\widehat\theta^{\,0}. 
\]

\item[v)]  {\bfseries Number of segregating sites in a sample}

To obtain the number of segregating sites in a sample of size $m$, we take $f_{uu}(0)=0$ and $f_{uv}(y)=g_m(y)$, where
\begin{equation*}
    g_m(y)
    = \sum_{k=1}^{m-1} \binom{m}{k} y^k(1-y)^{m-k}
    =1-y^m-(1-y)^m,\quad 0\le y\le 1,
\end{equation*}
is the probability that a sample of size $m\ge 2$ is polymorphic when drawn from a population with derived frequency $y$. Then, it holds $\int_0^{1/N}g_m(y)\,dy\sim O(1/N^2)$ and
\[
\int_{1/N}^1 y^{-1}g_m(y)\,dy\sim 
\int_0^1 y^{-1}g_m(y)\,dy
=\sum_{k=1}^{m-1} \frac{1}{k}.
\]
Hence
\begin{align*}\label{eq:upper_bound_segr_sites_sample}
    S_{N,L}^m &= \E_N \langle \Xc_\infty,f\rangle
     \sim  2\widehat\theta_\eff^{\,0} 
     \,\sum_{k=1}^{m-1} \frac{1}{k}
     \le 2{\widehat\theta}^{\,0}
      \,\sum_{k=1}^{m-1} \frac{1}{k}.
\end{align*}

\item[vi)] {\bfseries Pair-wise nucleotide differences}

The standard measure of genetic diversity in the population, typically denoted $\pi$, is the average number of pair-wise nucleotide differences normalized per site \citep{Nei1979}. 
For sample size $m$ we take $f_{uu}(0)=0$ and $f_{uv}(y)=h_m(y)$, where 
\begin{equation*}
    h_m(y)
    = \binom{m}{2}^{-1}\,\sum_{k=1}^{m-1} k(m-k)\binom{m}{k} y^k(1-y)^{m-k}
    = 2y(1-y).
\end{equation*}
Hence 
\[
\pi= \E_N \langle \Xc_\infty,f\rangle /L
= 2\,\frac{\widehat\theta_\eff^{\,0}}{L}
\le 2\,\frac{\widehat\theta^{\,0}}{L},
\]
which turns out to be the average mutation load per site and is the same as the expected proportion of segregating sites in a sample of size two, $S_{N,L}^2/L$.

\end{itemize}

\subsection{Allele frequency statistics and their upper bounds}

We are interested in the overall effect of directional selection acting on the functionals covered above and closely related quantities, as compared to their counterparts under neutral evolution.  Our main result shows that any presence of directional selection among the alleles essentially benefits monomorphic loci and constrains the number of polymorphic loci.

\begin{theorem}\label{thm:upper_bound}
We consider the Wright-Fisher graph model extended to $L$ loci with fixed, arbitrary parameters $(\theta_{uv})$ for mutation and $\{(\gamma_{uv}^j)\}_{1\le j\le L}$ for selection. 
\begin{itemize}
    \item[1)] With $\widehat\theta$, $\widehat\theta_\eff$, $\theta_\mathrm{min}$ and $\theta_\mathrm{max}$ defined in \cref{eq:thetahat,eq:thetahateff,eq:thetamax}, we have
    \[
        \frac{\theta_\mathrm{min}}
        {1+2L^{-1}(1+\ln N) \theta_\mathrm{min}}   
        \le \widehat\theta_\eff
        \le \widehat \theta \le \theta_\mathrm{max}.
    \]
    \item[2)] Let $f$ be a function on $D$ such that $f_{uv}(y)=f(y)\geq0$ is a function only of the frequency $y$ with $f(y)\to f(0+)\geq0$, $y\to0$, and $f(0)=0$. Then 
    \begin{align}\label{eq:frequency_only}  \nonumber
        0\le \E_N  \langle \Xc_\infty,f\rangle &\le 
        2 \widehat\theta_\eff \,\Big\{f(0+) + \int_{1/N}^1\frac{ f(y)}{y}\,dy\Big\}\\
        &\le 
        2 \widehat\theta \,\Big\{f(0+) + \int_{1/N}^1\frac{ f(y)}{y}\,dy\Big\}.
    \end{align}
    If, moreover, $\int_0^1 y^{-1} f(y)\,dy<\infty$, then
    \begin{equation}\label{eq:f_integrable}
        \E_N \langle \Xc_\infty,f\rangle \le 
        2\widehat\theta_\eff \int_0^1\frac{f(y)}{y}\,dy
       \le 
        2\widehat\theta \int_0^1\frac{f(y)}{y}\,dy. 
    \end{equation}
\end{itemize}
\end{theorem}
While the strength and direction of selection may vary arbitrarily within and between sites, \cref{thm:upper_bound} illustrates that the effect of selective forces on measures of genetic variation is only channeled through to the upper bounds via the average mutation rates $\widehat\theta_\eff$ and $\widehat\theta$, respectively. 
It is seen furthermore that the proportionality constant $\widehat\theta_\eff$ is contained inside an interval that does not depend on selection parameters, namely the interval formed by the leftmost and the rightmost estimate in \cref{thm:upper_bound}, 1).
As a corollary we observe that under the stronger assumption of homogeneous mutation rates, introduced in \cref{sec:multiple_loci}, then $\widehat\theta$ will be independent of any selective mechanisms, and the upper bounds in \cref{eq:frequency_only,eq:f_integrable} will coincide with the corresponding expressions for neutral evolution in \cref{lst:summary_statistics_neutral} ii). 

\begin{corollary}\label{cor:equal_mutation_rates}
For the case when the mutation rates are homogeneous over all vertices, i.e.\ $\theta_u=\theta$ for all $u\in\Vc$, then
\[
\widehat \theta =\widehat\theta^{\,0}=\theta
\]
and
\[
\frac{\theta}
{1+2L^{-1}(1+\ln N)\theta}
\le \widehat\theta_\eff \le \theta.
\]
Hence,
\[
 \E_N \langle \Xc_\infty,f \rangle
 \le 2\theta\, \Big\{f(0+) + \int_{1/N}^1\frac{ f(y)}{y}\,dy\Big\}
\]
and
\[
 \E_N \langle \Xc_\infty,f \rangle 
 \le 2\theta \int_0^1\frac{f(y)}{y}\,dy,
\]
respectively.
\end{corollary}

\begin{proof}[Proof of \cref{thm:upper_bound,cor:equal_mutation_rates}]
1) For each single locus $j$, 
\begin{align*}
\Omega_N^j & =1+\frac{2}{L}\sum_{u,v\in\Vc} 
   \eta_u^j\theta_{uv}(1+\ln N+K_{\gamma^j_{uv}})\\
   & = 1+\frac{2\widehat\theta^j}{L}(1+\ln N) 
   +\frac{2}{L}\sum_{u,v\in\Vc} 
   \eta_u^j\theta_{uv} K_{\gamma^j_{uv}}.
\end{align*}
Here, by rewriting the double sum over all vertices in $\Vc$ as the sum over all unordered pairs of vertices in $\Pc$ (see \cref{sec:stat_distr_undirected_edges}),
\begin{align*}
\sum_{u,v\in\Vc} 
   \eta_u^j\theta_{uv} K_{\gamma^j_{uv}}
   &=\sum_{\langle u,v \rangle\in\Pc} \{\eta_u^j\theta_{uv} K_{\gamma^j_{uv}}+\eta_v^j\theta_{vu} K_{\gamma^j_{vu}}\}\\
   &= \sum_{\langle u,v \rangle\in\Pc} \eta_u^j\theta_{uv}(1-e^{2\gamma_{uv}^j})K_{\gamma_{uv}^j}\le 0, 
\end{align*}
for every $\gamma_{uv}^j$. Hence $1\le \Omega_N^j\le 1+2\widehat\theta^j(1+\ln N)/L$ and therefore
\begin{equation*}
\begin{split}
    \widehat\theta\ge 
    \widehat\theta_\eff 
    &=\frac{1}{L}\sum_{j=1}^L \frac{\widehat\theta^j}{\Omega_N^j}\\
    &\ge  \frac{1}{L}\sum_{j=1}^L \frac{\widehat\theta^j}{1+2\widehat\theta^j(1+\ln N)/L}
    \ge \frac{\theta_\mathrm{min}}{1+2L^{-1}(1+\ln N) \theta_\mathrm{min}}.
\end{split}
\end{equation*}

\noindent
2) To prove the bound of $\E_N \langle \Xc_\infty,f \rangle$ in the second statement we take $f_{uu}(0)=f(0)=0$ in \cref{eq:mean_steadystate} and start from the representation
\begin{align*}
\E_N\langle \Xc_\infty,f\rangle
&=\sum_{j=1}^L \sum_{u,v\in\Vc} 
  \int_0^1 f(y) \mu^j_N(u,v,y)\,dy.
\end{align*}
As we apply Proposition \ref{prop:statapprox} it is convenient to have the auxiliary notation $J_{uv}(y)$ (only used in this proof)
\[
   J_{uv}(y)=\frac{
    1-e^{-2\gamma_{uv}(1-y)}}{2\gamma_{uv}(1-y)},
    \quad 0<y<1.
\]
Then
\begin{equation*}
\begin{split}
    &\int_0^1  f(y) \mu^j_N(u,v,y)\,dy
    = \frac{2\eta_u^j\theta_{uv}}{L\Omega^j_N} \Big\{ f(0+) +
    \int_{1/N}^1 \frac{f(y)}{y} \, \omega_{\gamma^j_{uv}} J_{uv}^j(y)\,dy\Big\}.
\end{split}
\end{equation*}
We partition the right hand side as
\begin{align*}
    \int_0^1  f(y) \mu^j_N(u,v,y)\,dy
 = \frac{2\eta_u^j\theta_{uv}}{L\Omega^j_N} 
 \Big\{ f(0+) +
    \int_{1/N}^1 \frac{f(y)}{y}\,dy\Big\}
    -R_N^j(u,v),
\end{align*}
with
\begin{equation*}   
    R_N^j(u,v) = \frac{2\eta_u^j\theta_{uv}}{L\Omega^j_N} \int_{1/N}^1 \frac{f(y)}{y} 
    \Big\{1-\omega_{\gamma_{uv}}^j J_{uv}^j(y) \Big\}\,dy.
\end{equation*}
Letting $R_N$ denote the sum
\[
    R_N=\sum_{j=1}^L \sum_{u,v\in\Vc} R_N^j(u,v),
\]
these considerations imply
\begin{align*}
    \E_N\langle \Xc_\infty,f\rangle
    &=\sum_{u,v\in\Vc} 
    \widehat\mu_u\,2\theta_{uv} 
    \Big\{ f(0+) + \int_{1/N}^1 \frac{f(y)}{y}\,dy\Big\}
    -R_N.
\end{align*}
To complete the proof it remains to show that $R_N\ge 0$. Now, 
\begin{equation*}   
    R_N = \sum_{j=1}^L  \frac{2}{L\Omega_N^j}
    \int_{1/N}^1 \frac{f(y)}{y}\sum_{u,v\in\Vc}
    \eta_u^j\theta_{uv} 
    \Big\{1-\omega_{\gamma^j_{uv}}
    J_{uv}^j(y)\Big\}\,dy.
\end{equation*}
Thus, it suffices to show, for each site $j$ and each frequency $y$,
\[
\sum_{u,v\in\Vc}
    \eta_u^j\theta_{uv}
    \big\{1-\omega_{\gamma^j_{uv}}
    J_{uv}^j(y)\big\} \ge 0.
\]
By rewriting the double summation as sum over all unordered pairs $\langle u,v \rangle\in\Pc$, the previous inequality has the equivalent representation
\begin{align*}
    \sum_{\langle u,v \rangle\in\Pc}\Big\{ 
    \eta_u^j\theta_{uv} 
    \big\{1-\omega_{\gamma^j_{uv}} J_{uv}^j(y)\big\}
    + \eta_v^j\theta_{vu} 
    \big\{1-\omega_{\gamma^j_{vu}} J_{vu}^j(y)\big\} \Big\}\ge 0.        
\end{align*}
By applying the detailed balance equation to each site and each edge, the task is to show
\begin{align*}
    \sum_{\langle u,v \rangle\in\Pc} 
    \eta_u^j\theta_{uv}
    \Big\{ 
    \big\{1 -\omega_{\gamma^j_{uv}}J_{uv}^j(y)\big\}
    + \frac{\omega_{\gamma_{uv}^j}}
    {\omega_{\gamma_{vu}^j}}
    \big\{1-\omega_{\gamma^j_{vu}}
    J_{vu}^j(y)\big\} \Big\}\ge 0.  
\end{align*}
Equivalently,
\begin{align*}
    \sum_{\langle u,v \rangle\in\Pc} 
    \eta_u^j\theta_{uv}\omega_{\gamma^j_{uv}} R^j_{uv}\ge 0,
    \quad 
    R^j_{uv}=\omega_{\gamma^j_{uv}}^{-1}-J_{uv}^j(y)
    + \omega_{\gamma_{vu}^j}^{-1}-J_{vu}^j(y). 
\end{align*}
Next we use the anti-symmetric relation (\ref{def:timereversible_sel}) for the selection coefficients. If we take a fixed site $j$ and an edge which connects two vertices, $u$ and $v$ say, and let $\gamma=\gamma^j_{uv}=-\gamma^j_{vu}$ be one of the relevant selection coefficients, then
it is straightforward to check that $R^j_{uv}$ is indeed nonnegative for any signed parameter $\gamma$ and $0<y<1$,
\[
    R^j_{uv}=\frac{e^{2\gamma}-e^{-2\gamma}}{2\gamma}
    -\frac{e^{2\gamma(1-y)}-e^{-2\gamma(1-y)}}{2\gamma(1-y)}\ge 0.
\]
This verifies the claim $R_N\ge 0$ and yields 
\begin{equation*}
\begin{split}
    \E_N\langle \Xc_\infty,f\rangle
    &\leq\sum_{j=1}^L \sum_{u,v\in\Vc} \frac{2\eta_u^j\theta_{uv}}{L\Omega^j_N} 
    \Big\{ f(0+) + \int_{1/N}^1 \frac{f(y)}{y}\,dy\Big\} \\
    &= 2\widehat\theta_\eff \Big\{ f(0+) + \int_{1/N}^1 \frac{f(y)}{y}\,dy\Big\}.
\end{split}
\end{equation*}
We note that this proof actually provides a more general result for functions $f$ on $D$ that are not independent of $u$ and $v$ but fulfill $f_{uv}(y)=f_{vu}(y)\ge 0$. Then 
\[
    0\le \E_N \langle \Xc_\infty,f\rangle \le 
    2 \sum_{u,v\in \Vc}  \widehat\mu_u\theta_{uv}  \Big\{f_{uv}(0+) + 
    \int_{1/N}^1\frac{ f_{uv}(y)}{y}\,dy\Big\}.
\]
The other statement in part 2) of \cref{thm:upper_bound} is straightforward under the additional assumption. For \cref{cor:equal_mutation_rates}, a simple calculation verifies that $\widehat\theta=\widehat\theta^{\,0}=\theta$ if $\theta_u=\theta$ for every $u\in\Vc$. The rest of the corollary follows directly from \cref{thm:upper_bound}.
\end{proof}

\subsection{The number of segregating sites and genetic diversity} 

We are now in position to consider concrete measures of genetic variation in a population under the general model with selection and compare with the known properties of these measures for neutral evolution as listed in \cref{lst:summary_statistics_neutral}. \Cref{thm:upper_bound} provides general estimates valid for arbitrary coefficients of directional selection. First, \Cref{thm:upper_bound} applied with the functions $f^\partial$ and $f^\circ$ yield bounds which directly relate to the listed items iii) and iv) of \cref{lst:summary_statistics_neutral}. In particular, the expected number of segregating sites under selection satisfies 
\begin{equation}\label{eq:upper_bound_nmb_segr_sites}
\E_N \langle \Xc_\infty,f^\circ\rangle
    \le 2\widehat \theta_\eff \,(1+\ln N),
\quad
\widehat\theta_\eff=\sum_{u\in\Vc}\widehat\mu_u\theta_u.
\end{equation}
The parallel result for the number of segregating sites in a sample, i.e.\ \cref{thm:upper_bound} applied with the function $f$ specified in item v), reads
\begin{align}\label{eq:upper_bound_segr_sites_sample}
    S_{N,L}^m &= \E_N \langle \Xc_\infty,f\rangle
     \le 2\widehat\theta_\eff \,\sum_{k=1}^{m-1} \frac{1}{k}.
\end{align}
Similarly, the expected genetic diversity in the population satisfies $\pi\leq 2\widehat\theta_\eff/L$, which extends vi) of \cref{lst:summary_statistics_neutral}.

For specific functions $f$  we may of course extract more detailed information in addition to the upper bounds discussed here. It is again convenient to carry out summation over unordered pairs of graph vertices. 
For this, let us assume that $f_{uu}(0)=0$, $f_{uv}(y)=f(y)$, and $\int_0^1 y^{-1}f(y)\,dy<\infty$. Then
\begin{align*}
    &\E_N \langle \Xc_\infty,f\rangle =  \\
    &\quad 
    \sum_{j=1}^L \sum_{\langle u,v\rangle\in\Pc} \frac{2\eta_u^j\theta_{uv}\omega_{\gamma^j_{uv}}} {L\Omega_N^j} 
    \int_0^1 \frac{f(y)}{y(1-y)} 
    \Big(\frac{e^{2\gamma_{uv}^j(1-y)}
    -e^{-2\gamma_{uv}^j(1-y)}}{2\gamma^j_{uv}}\Big)\,dy.
\end{align*}
In particular, for $f(y)=2y(1-y)$, 
\begin{align*}
\pi
= \E_N \langle \Xc_\infty,f\rangle/L
= \sum_{\langle u,v\rangle\in\Pc} \frac{2\theta_{uv}}{L}\frac{1}{L} 
\sum_{j=1}^L \frac{2}{\Omega_N^j} \frac{\eta_u^j}{\omega_{\gamma_{vu}^j}}.
\end{align*}

The functional $f^\circ$ to obtain the expected number of segregating sites does not fulfill the condition $\int_0^1 y^{-1}f^\circ(y)\,dy<\infty$. Nevertheless we obtain an explicit representation of the expected number of segregating sites. Since the probability that a single site $j$ is polymorphic is $(\Omega_N^j-1)/\Omega_N^j$, the expected number of segregating sites in a sequence of length $L$ is
\begin{equation*}
\begin{split}
    &\E_N\langle\Xc_\infty,f^\circ\rangle
    = \sum_{j=1}^L \frac{\Omega_N^j-1}{\Omega_N^j}\\
    &= \sum_{j=1}^L \frac{2\sum_{\langle u,v\rangle\in\Pc}\eta_u^j\theta_{uv}\{1+\ln N+K_{\gamma_{uv}^j} + e^{2\gamma_{uv}^j}(1+\ln N-K_{\gamma_{uv}^j})\}}{L + 2\sum_{\langle u,v\rangle\in\Pc}\eta_u^j\theta_{uv}\{1+\ln N+K_{\gamma_{uv}^j} + e^{2\gamma_{uv}^j}(1+\ln N-K_{\gamma_{uv}^j})\}}.
\end{split}
\end{equation*}

\section{Discussion}

We have set up a multi-allele, multi-locus Wright-Fisher graph model to derive rigorous upper bounds for a wide class of summary statistics of genetic variation in \cref{thm:upper_bound}. For any representative measure in this class the upper bound is a multiple of the average effective mutation rate $\widehat\theta_\textrm{eff}$. The multiplicative factor is independent of directional selection and purely depends on the measure of genetic diversity. Hence, mutation and directional selection only affect the upper bound through $\widehat\theta_\textrm{eff}$ or $\widehat\theta$. To obtain selection-independent upper bounds for arbitrary mutation rates, $\widehat\theta$ can be replaced with e.g.\ $\theta_\textrm{max}$.
The additional observation in \cref{cor:equal_mutation_rates} that homogeneous mutation rates make $\widehat\theta$ independent of directional selection shows that the upper bounds are the same as those for neutral evolution, and hence verifies the general presumption that directional selection reduces genetic variation.

There exists a number of deterministic models to verify the reduction of genetic variation due to directional selection \citep{Feldman1971,Novak2017,Pontz2020}, also referred to as "constant frequency-independent selection". These models are based on replicator equations, that were initially used by \citet{Feldman1971} in this field. Within this deterministic modeling approach analytical results on the interactions between loci due to physical linkage and/or epistasis can be derived. On the other hand, mutation and genetic drift as additional evolutionary forces are often not taken into account. 
Here, we take a complementary approach and incorporate mutation and genetic drift, while ignoring interactions among loci. In the following section we will illustrate the relevance of this setting for studying the interaction between mutation and fixation bias. Finally, we demonstrate that a distribution of fitness effects of e.g.\ protein-coding sequences naturally derives from our modeling setup.

\subsection{Interaction between mutation and fixation bias}

We say that there is a mutation bias between two allelic types $u,v\in\Vc$, if mutation from one type to the other occurs more often than in the reverse direction, i.e.\ if $\theta_{uv}\neq\theta_{vu}$, and there is fixation bias between $u$ and $v$ whenever 
$\gamma_{uv}\not=0$. Thus, fixation bias comprises non-zero directional selection. 
In addition to selection, fixation bias can be caused by biased gene conversion, a mechanism which is probabilistically equivalent to directional selection \citep{Nagylaki1983}. 
The Wright-Fisher graph model allows to investigate the combined impact of mutation bias and fixation bias on the stationary distribution and population genetic measures derived as functionals of the stationary distribution. Relevant combinations of mutation bias and fixation bias may act in opposite directions and hence counterbalance their influence on the stationary distribution or act in the same direction and thus reinforce each other.

First of all we notice that given a collection of biased or unbiased mutation  rates $(\theta_{uv})$, any desired set of boundary distributions $(\eta^j_u)_{1\le j\le L}$ on $L$ loci is obtained from the detailed balance equations (\ref{def:detailed_balanceN}) by putting
\begin{equation}\label{eq:relationship_mutation_selection}
    \gamma^j_{uv}=\frac{1}{2}\ln\Big(\frac{\eta_v^j\theta_{vu}}{\eta_u^j\theta_{uv}}\Big).
\end{equation}
As a reference case, the assignment $2\gamma^j_{uv}=\ln (\theta_{vu}/\theta_{uv})$ yields uniform distributions over all monomorphic states. 
These relations are well-known and appear in studies of mutation-selection models such as \cite{Halpern1998,McVean1999}.
Next, including polymorphic loci, for a scenario with homogeneous  mutation rates and prescribed selection parameters, 
\Cref{cor:equal_mutation_rates} shows that genetic variation overall behaves much in the same way as for the case with no fixation bias. Homogeneous mutation, however, is arguably not necessarily realistic for genetic data, except perhaps on graphs of constant degree, i.e., graphs with the same number of edges attached in each vertex.  
The graph in \cref{fig:stopcodons} for the three stop codons provides a suitable example of a graph of mixed degree. A particular mutation bias would be required to give the relation $\theta_{TAG\to TAA}=\theta_{TAA\to TAG}+\theta_{TAA\to TGA}=\theta_{TGA\to TAA}$ for homogeneous mutation.
\begin{figure}[!ht]
    \centering
    \includegraphics{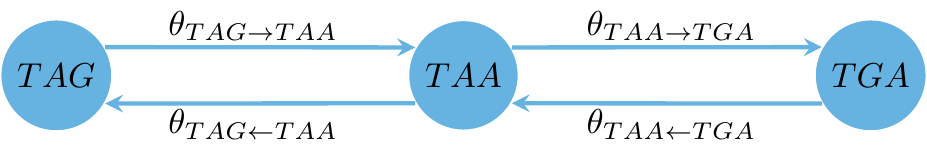}
    \caption{The graph for stop codons. A specific mutation bias is required for homogeneous mutation rates over all stop codons.
    }
    \label{fig:stopcodons}
\end{figure}
Hence, whenever the total mutation rates among types differ, it is worth studying the combined impact of mutation bias and fixation bias on the upper bounds in \cref{thm:upper_bound}. 

\paragraph{Four nucleotide model}

The interaction of mutation and fixation bias is of interest in the four nucleotide model (\cref{fig:AFSs}A, nodes representing nucleotides), as mutation rates are frequently found to be different among the nucleotides \citep{Stoltzfus2015,Long2018}.
Both biases impact for example the sequence content, the abundance of each nucleotide in the sequence, which can be obtained from \cref{eq:mean_steadystate}. 

The four nucleotide model can be reduced to a model with two types by grouping the nucleotides into two classes. A natural classification arises when studying nucleotide composition in the double-stranded DNA and considering A:T base-pairs as one group and C:G base-pairs as another. This classification, weak (A and T) and strong (C and G) bases, is commonly used to describe gBGC \citep{Duret2009,Mugal2015}. The fixation bias towards GC over AT nucleotides in the presence of gBGC interacts with the mutation bias between the two classes, which acts in the opposite direction in several taxa, i.e.\ mutations from GC to AT occur much more frequently than mutations from AT to GC \citep{Long2018}. 
This illustrates that the two-type model can be relevant to describe multiple alleles that can be classified into two types.
The traditionally studied two-type case is that of identically distributed loci, $\gamma^j=\gamma$ for each $j=1,\ldots,L$ \citep{Wright1931,Li1987,Bulmer1991,McVean1999}.
The resulting three-parameter situation is frequently applied in models of gBGC \citep{Muyle2011,Maio2013,Lachance2014} and analytical results of the biallelic mutation-selection-drift model have been derived \citep{Vogl2012,Vogl2015}.

\paragraph{Interaction of mutation and fixation bias in a two-type model}

We consider the graph with two types, $u$ and $v$, in which the dynamics at a fixed locus $j$ are determined by three parameters $\theta_u$, $\theta_v$ and $\gamma^j=\gamma_{uv}^j$, $j=1,\ldots,L$.
Clearly, by (\ref{def:detailed_balanceN}), 
\begin{equation}\label{eq:eta_two_types}
    (\eta^j_u,\eta^j_v) = \Big(\frac{\theta_v} {\theta_v+\theta_u e^{2\gamma^j}}, 
    \frac{\theta_u e^{2\gamma^j}}{\theta_v+\theta_u e^{2\gamma^j}}\Big).
\end{equation}
Under the assumption $\gamma^j=\gamma$ for all sites, we may summarize the effect of varying $\theta_u$, $\theta_v$, and $\gamma$ in \cref{eq:eta_two_types}, as
\begin{equation*}
    \frac{\eta_u}{\eta_v} = \frac{\theta_v}{\theta_u} e^{-2\gamma} 
    \begin{cases}
    >1 & \mathrm{if}\, \gamma<\frac{1}{2}\ln(\theta_v/\theta_u),\\
    =1 & \mathrm{if}\, \gamma=\frac{1}{2}\ln(\theta_v/\theta_u),\\
    <1 & \mathrm{if}\, \gamma>\frac{1}{2}\ln(\theta_v/\theta_u).
    \end{cases}
\end{equation*}
To illustrate the results obtained in \cref{thm:upper_bound}, the upper bounds are controlled by 
\begin{align*}
\widehat\theta_\eff 
=\frac{1}{L}\sum_{j=1}^L
\frac{1}{\Omega_N^j}
\frac{\theta_u\theta_v(1+e^{2\gamma^j})}{\theta_v+\theta_u e^{2\gamma^j}}
\le \widehat\theta 
=\frac{1}{L}\sum_{j=1}^L
\frac{\theta_u\theta_v(1+e^{2\gamma^j})}{\theta_v+\theta_u e^{2\gamma^j}}.
\end{align*}
To some degree these averaged mutation rates still depend on the selection coefficients. However, 
\begin{equation*}
\widehat\theta_\eff
\le \theta_{\max}
\,\frac{1}{L}
\sum_{j=1}^L\frac{1}{\Omega_N^j}\le \theta_{\max},
\end{equation*}
where the averaging sum is an estimate of the fraction of monomorphic loci in the sequence.  It is of interest nonetheless to compare with the harmonic mean of the mutation rates appearing under neutral evolution, namely
\begin{align*}
\widehat\theta^{\,0}_\eff 
=\frac{1}{\Omega_N}
\frac{2\theta_u\theta_v}{\theta_v+\theta_u}
\le  \widehat\theta^{\,0} 
=\frac{2\theta_u\theta_v}{\theta_v+\theta_u}.
\end{align*}
Considering the ratio $\widehat\theta/\widehat\theta^{\,0}$ for $\gamma^j=\gamma$, $j=1,\ldots,L$, we obtain the relations
\begin{equation*}
    \frac{\widehat\theta}{\widehat\theta^{\,0}}
    = \frac{(\theta_u+\theta_v)(1+e^{-2\gamma})}{2(\theta_u+\theta_v e^{-2\gamma})}
    \begin{cases}
    >1 & \mathrm{if }\, (\theta_u>\theta_v \land \gamma<0) \lor (\theta_u<\theta_v \land \gamma>0),\\
    =1 & \mathrm{if }\, \gamma=0 \lor \theta_u=\theta_v,\\
    <1 & \mathrm{if }\, (\theta_u<\theta_v \land \gamma<0) \lor (\theta_u>\theta_v \land \gamma>0).
    \end{cases}
\end{equation*}
This implies that $\widehat\theta>\widehat\theta^{\,0}$ if mutation bias is opposing fixation bias and $\widehat\theta<\widehat\theta^{\,0}$ if mutation and fixation biases enhance each other. 

\begin{figure}[!ht]
    \centering
    \includegraphics[width=\textwidth]{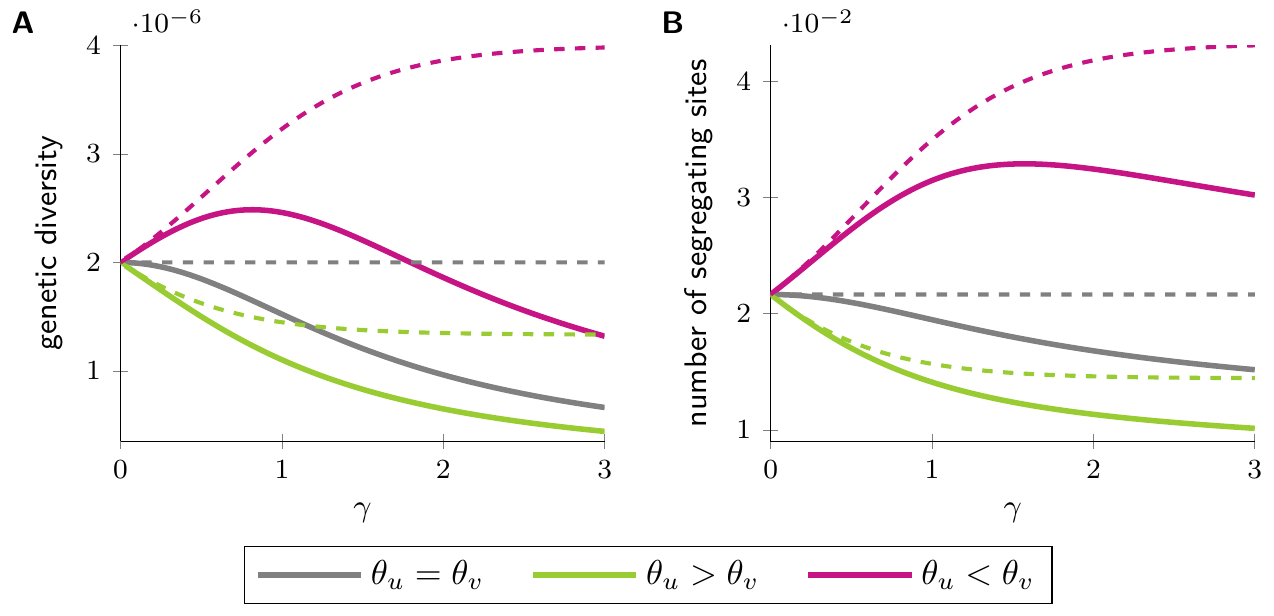}
    \caption{Genetic diversity $\pi$ (solid lines) and its upper bound $2\widehat\theta_\mathrm{eff}/L$ (dashed lines) in panel A and the expected number of segregating sites and its upper bound $2\widehat\theta_\mathrm{eff} (1+\ln N)$ (dashed lines) in panel B for different combinations of mutation parameters in a two-type model with equal selection coefficient among loci. Gray curves: $\theta_u=\theta_v=1.5\cdot10^{-3}$, green curves: $\theta_u=3\cdot10^{-3}>\theta_v=1\cdot10^{-3}$, pink curves: $\theta_u=1\cdot10^{-3}<\theta_v=3\cdot10^{-3}$. Other parameters: population size $N=500$ and number of loci $L=1500$. 
    }
    \label{fig:genetic_diversity}
\end{figure}

Such insights can be used together with the results of \cref{thm:upper_bound}, for instance, considering the case of genetic diversity (\cref{fig:genetic_diversity}), 
\begin{equation*}
\begin{split}
    \pi &= 
    \frac{1}{L\Omega_N}
    \frac{4\theta_u\theta_v} {\theta_v+\theta_u e^{2\gamma}} 
    \frac{e^{2\gamma}-1}{2\gamma}
    \le 2\frac{\widehat\theta_\eff}{L}=
    \frac{1}{L\Omega_N}
    \frac{2\theta_u\theta_v(1+e^{2\gamma})}{\theta_v+\theta_u e^{2\gamma}}.
\end{split}
\end{equation*}
Without mutation bias or with a mutation bias that enhances the fixation bias, genetic diversity decreases monotonically as selection becomes stronger (gray and green solid curves in \cref{fig:genetic_diversity}A). If mutation bias counteracts fixation bias, genetic diversity first increases in the weak selection regime compared to neutral evolution until a maximum is reached for an intermediate selection coefficient, and decreases thereafter for stronger selection (pink solid curve in \cref{fig:genetic_diversity}A). A similar behavior is observed and discussed in \citet{McVean1999}. 
The upper bound (dashed lines in \cref{fig:genetic_diversity}A) is constant for equal mutation rates, decreases monotonically if mutation and fixation bias reinforce each other, and increases monotonically for counterbalancing biases. 
The behavior of the expected number of segregating sites and its upper bounds under the different combinations of mutation rates is very akin to the curves for genetic diversity (\cref{fig:genetic_diversity}B). 

The scenario depicted here where all loci have equal selective pressure that can become arbitrarily large is rather artificial. 
In many taxa the genome-wide average of fixation bias in gBGC takes a value in the weak selection regime \citep{Maio2013,Glemin2015,Galtier2018,Boman2021}. Likewise, according to the nearly neutral theory \citep{Ohta1973,Ohta1976,Ohta1992} polymorphisms segregate in a population if selection is neutral or nearly neutral. In this selection regime the upper bounds capture the behavior of the measure of genetic variation well.
Only when selection becomes strong, the upper bounds become more generous. However, strong selection immediately removes genetic variation and consequently, the interaction of mutation and fixation bias in the strong selection regime is less relevant when considering a large collection of loci. 

\subsection{Distribution of fitness effects}\label{sec:DFE}

The present model equipped in each locus with a static fitness landscape as in \cref{def:fitness_sel} can be applied for example to protein-coding sequences. 
In this case every locus represents a nucleotide triplet where the fitness of each possible codon type is particular to the specific locus. 
However, the number of fitness parameters required to capture the selection effects over many loci and all codons may quickly grow out of hand \citep[for a recent review on this topic see\ ][]{Youssef2021}. 
An alternative view of modeling natural selection is that of choosing fitness parameters from a representative distribution. In this approach, the distribution of fitness effects (DFE) should reduce the parameter space but preserve some of the characteristic features.     
Generally, the DFE is composed of a distribution of negative selection coefficients, a distribution of positive selection coefficients, and a proportion of selection coefficients at zero representing neutral evolution. 
The steady-state of the process $(X_t)_{t\ge 0}$ in our modeling framework allows deriving a DFE which is informative about the fraction of mutations in equilibrium that are beneficial, deleterious and neutral, respectively.

\paragraph{The distribution of fitness effects of novel mutations} 

The common understanding of a DFE is the distribution of fitness effects of all novel mutations that occur in a population \citep{EyreWalker2007}. As we apply a boundary mutation model, the weights of the selection coefficients are given by the steady states of the Wright-Fisher graph model on the boundary, $\eta_u^j/\Omega_N^j, u\in\Vc$, and the corresponding mutation intensities, $\theta_{uv}, u,v\in\Vc$. We may assume that the mutation intensities have been reduced in advance to discount for the presence of any strongly deleterious mutations. Hence, the DFE on the graph is a discrete probability distribution function on the real line,
\begin{equation}
    H_{\textrm{dfe}}(\gamma)
    =\widehat\theta_\eff^{-1} \, \frac{1}{L} \sum_{j=1}^L \sum_{u,v\in\Vc} \frac{\eta_u^j}{\Omega_N^j}  \theta_{uv} 1_{\{\gamma_{uv}^j\le \gamma\}}, \quad -\infty<\gamma<\infty,
\end{equation}
with jumps at each of the finite number of values $\gamma_{uv}^j, u,v\in\Vc, j=1,\ldots,L$. 
In \cref{prop:DFE} we obtain some key properties of the DFE.

\begin{proposition}\label{prop:DFE}
For any choice of selection coefficients on the Wright-Fisher graph model, 
\begin{itemize}
\item[i)] the probability weight for each negative selection coefficient, $\gamma_{uv}^j<0$, is always larger or equal the weight of the corresponding positive value, that is,
\[
H_\mathrm{dfe}(\gamma_{uv}^j)-H_\mathrm{dfe}(\gamma_{uv}^j-)
\ge H_\mathrm{dfe}(-\gamma_{uv}^j)-H_\mathrm{dfe}(-\gamma_{uv}^j-),
\]
\item[ii)] the average selection load is deleterious, in the sense
\[
\widehat \gamma= \widehat\theta_\eff^{-1}\, \frac{1}{L} \sum_{j=1}^L \sum_{u,v\in\Vc} \frac{\eta_u^j}{\Omega_N^j} \theta_{uv}  \gamma_{uv}^j 
\leq 0,
\]
\item[iii)] the total contribution to positive selection is
\[
1-H_\mathrm{dfe}(0)<1/2.
\]
\end{itemize}
\end{proposition}

\begin{proof}
To prove i) we use the alternative representation of $H_\mathrm{dfe}$ summing over pairs of vertices and apply detailed balance on the edges between them, 
\begin{align*}
    H_\mathrm{dfe}(\gamma)
    &=\widehat\theta_\eff^{-1} \frac{1}{L}\sum_{j=1}^L
    \sum_{\langle u,v \rangle\in\Pc} \frac{1}{\Omega_N^j}
    \Big\{\eta_u^j \theta_{uv} 1_{\{\gamma_{uv}^j\le \gamma\}} + \eta_v^j \theta_{vu} 1_{\{\gamma_{vu}^j\le \gamma\}} \Big\}\\
    &= \widehat\theta_\eff^{-1} \frac{1}{L}\sum_{j=1}^L
    \sum_{\langle u,v \rangle\in\Pc} \frac{\eta_u^j}{\Omega_N^j} \theta_{uv} \Big\{1_{\{\gamma_{uv}^j\le \gamma\}} + e^{2\gamma_{uv}^j}1_{\{-\gamma_{uv}^j\le \gamma\}} \Big\}.
\end{align*}
For each fixed negative selection coefficient $\gamma=\gamma_{uv}^j<0$, we conclude that the jump size $H_\mathrm{dfe}(\gamma)-H_\mathrm{dfe}(\gamma-)$ is proportional to $1$, while the accompanying jump on the side of positive selection coefficients, $H_\mathrm{dfe}(-\gamma)-H_\mathrm{dfe}(-\gamma-)$, equals the same proportionality constant times $e^{2\gamma}$. Hence property i) follows since $1\geq e^{2\gamma}$.
For claim ii) we have, similarly,
\begin{align*}
    \widehat\gamma
    &=\widehat\theta_\eff^{-1} \frac{1}{L}\sum_{j=1}^L \sum_{\langle u,v \rangle\in\Pc} \frac{\eta_u^j}{\Omega_N^j} \theta_{uv} \gamma_{uv}^j \{1-e^{2\gamma^j_{uv}}\}.
\end{align*}
This quantity is nonpositive, since $\gamma_{uv}^j(1-e^{2\gamma_{uv}^j}) \leq 0$ for all $\gamma_{uv}^j, j=1,\ldots,L$. It follows that $\widehat\gamma\le 0$. 
Claim iii) is merely a rephrased conclusion of i).
\end{proof}

To help interpret the shape of $H_\mathrm{dfe}(\gamma)$, let us suppose that the $\gamma$-values for negative selection can be well approximated by a continuous distribution with density function $g(x)$, $x<0$, on the negative half line. It then follows from the proof of \cref{prop:DFE} that the associated positive $\gamma$-values have a density proportional to $g(-x) e^{-2x}$, $x>0$. The full approximation on the real line is then obtained by normalizing the contributions for $x<0$, $x>0$ and possibly an atom for neutral mutations at $x=0$. \Cref{fig:dfe} displays the resulting densities for the case of $g$ being an exponential, $g(x)=e^{-|x|}$ (red), or a gamma distribution, $g(x)=|x|^{a-1}e^{-|x|}/\Gamma(a)$ (blue for $a=2$ and light blue for $a=0.15$). 
\begin{figure}
    \centering
    \includegraphics[width=0.7\textwidth]{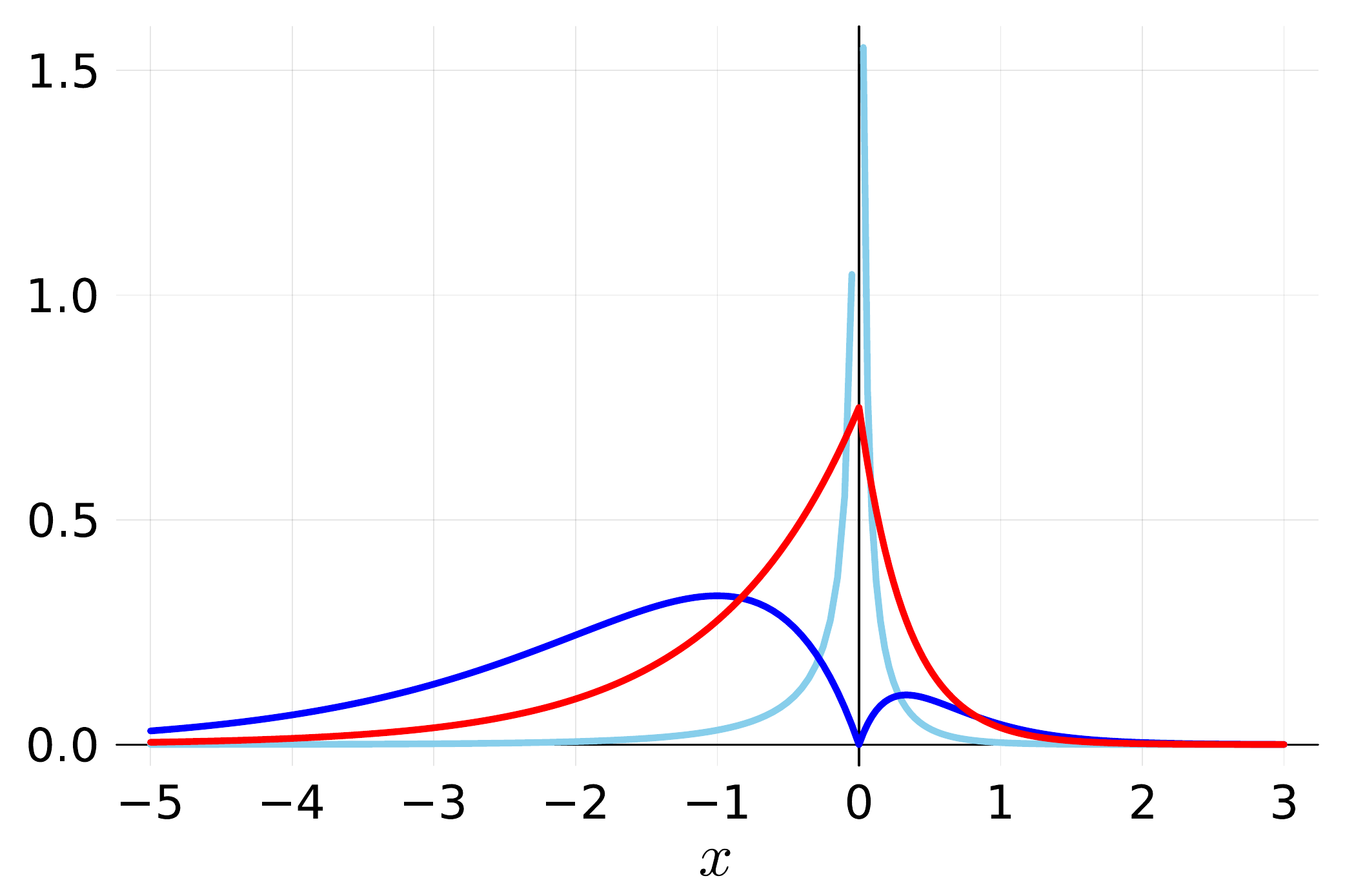}
    \caption{Approximative density of the DFE based on an exponential distribution with parameter 1 (red), a gamma distribution with shape parameter 2 and scale parameter 1 (blue), and a gamma distribution with shape parameter 0.15 and scale parameter 1 (light blue) for the negative selection coefficients.
    The average selection load is $\widehat\gamma=-2/3$ in case of the exponential distribution and $\widehat\gamma=-26/15$ and $\widehat\gamma\approx-0.058$, respectively, for the gamma distributions.
    }
    \label{fig:dfe}
\end{figure}

Moreover, \cref{prop:DFE} provides an approximation of the present model, where ancestral and derived alleles are kept distinct but not the genetic types. We consider the general Wright-Fisher graph model with average mutation load $\widehat\theta_\eff$ and DFE represented by $H_\mathrm{dfe}(\gamma)$. In each locus, lump together all boundary states into one generic vertex state $0$, which represents "ancestral".  Attach to $0$ a single, directed edge $e_0$ of length $1$ which returns to  $0$ at the endpoint. At exponential mutation rate $\widehat\theta_\eff$ a derived allele appears at frequency $1/N $ on the edge, following the path of a Wright-Fisher diffusion with selection coefficient drawn randomly from the distribution $H_\mathrm{dfe}$. At the time of extinction or fixation the derived returns to the state of ancestral. The same dynamics applies independently over the $L$ loci. 
The unfolded AFS of derived alleles in steady-state should then be well approximated by
\[
2\widehat\theta_\eff \int_\R \omega_\gamma \frac{1-e^{-2\gamma(1-y)}}{2\gamma y(1-y)}\, H_\mathrm{dfe}(d\gamma),\quad 0<y<1.
\]
The category of probabilistic models briefly described here is known as the Poisson random field approach in population genetics, see e.g.\ \citet{Sawyer1992,Mugal2014,Kaj2016}. \Cref{prop:DFE} provides general support for the Poisson random field, and \cref{prop:DFE} ii) even justifies the further simplified approximation where $H_\mathrm{dfe}$ is a one-point distribution with unit mass on a fixed $\gamma=\widehat\gamma<0$ for each mutation. 

\paragraph{The distribution of fitness effects of polymorphisms}

While in our setting mutation is linked to the boundary states, selection naturally acts on the interior of the graph. This suggests introducing the relevant distribution function for selection coefficients of segregating polymorphisms
\begin{equation}\label{def:Hgamma}
    \Tilde{H}_\mathrm{pdfe}(\gamma)
    =\frac{1}{L}\sum_{j=1}^L
    \sum_{u,v\in\Vc}  1_{\{\gamma_{uv}^j\le \gamma\}}
    \int_0^1 \mu^j_N(u,v,y)\,dy,\quad -\infty<\gamma<\infty.
\end{equation}
The function $\Tilde{H}_\mathrm{pdfe}$ is a (improper) discrete distribution on the real line with a finite number of jumps at each of the values $\pm\gamma_{uv}^j$ of the graph. Since  
\[
\Tilde{H}_\mathrm{pdfe}(\gamma)\to \Tilde{H}_\mathrm{pdfe}(\infty)=
\frac{1}{L}\sum_{j=1}^L\frac{\Omega^j_N-1}{\Omega^j_N} <1,\quad
\]
as $\gamma\to\infty$, the distribution is defect. Conditioning on polymorphic states yields a proper probability distribution $H_\textrm{pdfe}(\gamma)= \Tilde{H}_\mathrm{pdfe}(\gamma)/\Tilde{H}_\mathrm{pdfe}(\infty)$. 
It can be shown that \cref{prop:DFE} holds analogously for $H_\textrm{pdfe}(\gamma)$.

\begin{proof}[Proof of \cref{prop:DFE} for $H_\textrm{pdfe}$]
Summing over pairs of vertices, $H_\mathrm{pdfe}$ can be represented as
\begin{align*}
    H_\mathrm{pdfe}(\gamma)
    &=\Tilde{H}_\mathrm{pdfe}(\infty)^{-1} \frac{1}{L}\sum_{j=1}^L
    \sum_{\langle u,v \rangle\in\Pc}\frac{2}{L\Omega^j_N}
    \eta_u^j\theta_{uv}\Big\{(1+\ln N +K_{\gamma_{uv}^j}) 1_{\{\gamma_{uv}^j\le \gamma\}}\\
    &\hskip 3cm  +(1+\ln N -K_{\gamma_{uv}^j}) e^{2\gamma_{uv}^j}1_{\{-\gamma_{uv}^j\le \gamma\}}
    \Big\}.
\end{align*}
For each fixed negative selection coefficient $\gamma=\gamma_{uv}^j<0$, the jump size $H_\mathrm{pdfe}(\gamma)-H_\mathrm{pdfe}(\gamma-)$ is proportional to $1+\ln N +K_\gamma$, whereas the jump on the side of positive selection coefficients, $H_\mathrm{pdfe}(-\gamma)-H_\mathrm{pdfe}(-\gamma-)$, equals the same proportionality constant times $(1+\ln N -K_\gamma)e^{2\gamma}$. Hence property i) in \cref{prop:DFE} for $H_\mathrm{pdfe}$ follows if we can show
\[
1+\ln N+K_\gamma\ge (1+\ln N -K_\gamma)e^{2\gamma},
\]
that is
\begin{equation} \label{eq:Kg_inequality}
-K_\gamma  
\le \frac{1-e^{2\gamma}}{1+e^{2\gamma}}
 \,(1+\ln N),\quad \gamma<0,
\end{equation}
or, equivalently,
\begin{equation} \label{eq:Kg_inequality_plus}
K_\gamma  
\le \frac{1-e^{-2\gamma}}{1+e^{-2\gamma}}
 \,(1+\ln N),\quad \gamma>0.
\end{equation} 
For $0<\gamma<1$, using \cref{rem:K}, we have e.g. 
\[
K_\gamma\le \gamma 
\le\frac{1-e^{-2\gamma}}{1+e^{-2\gamma}}
 \,(1+\ln 2)
\]
and hence (\ref{eq:Kg_inequality_plus}) holds for $N\ge 2$.
Let $C$ be a constant such that $|\gamma_{uv}^j|\le C$ for all $1\le j\le L$ and $u,v\in \Vc$. 
For $1\le \gamma\le C$, by Remark \ref{rem:K},  
\[
K_\gamma
\le
\gamma_e+\ln 2\gamma 
\le \frac{1-e^{-2\gamma}}{1+e^{-2\gamma}} +\ln 2\gamma
\le \frac{1-e^{-2\gamma}}{1+e^{-2\gamma}}\big(1+\ln (4 C^2)\big),
\]
and hence (\ref{eq:Kg_inequality_plus}) holds for $N\ge 4C^2$.

For claim ii) we have, 
\begin{align*}
    \widehat\gamma
    &=\Tilde{H}_\mathrm{pdfe}(\infty)^{-1}\frac{1}{L}\sum_{j=1}^L
    \sum_{\langle u,v \rangle\in\Pc}\frac{2}{L\Omega^j_N}
    \eta_u^j\theta_{uv}\Big\{(1+\ln N +K_{\gamma^j_{uv}}) \gamma_{uv}^j \\
    &\hskip 3cm  -(1+\ln N -K_{\gamma^j_{uv}}) \gamma_{uv}^j e^{2\gamma^j_{uv}}
         \Big\}.
\end{align*}
The quantity in curly brackets is nonpositive if
\[
\gamma K_\gamma \le 
\gamma \,\frac{1-e^{-2\gamma}}{1+e^{-2\gamma}}\,(1+\ln N),
\]
which is true for all $\gamma=\gamma_{uv}^j, j=1,\ldots,L$, and $N$ sufficiently large by the inequalities (\ref{eq:Kg_inequality}) and (\ref{eq:Kg_inequality_plus}). Hence, it follows that $\widehat\gamma\le 0$. 
\end{proof}

In conclusion, \cref{prop:DFE} demonstrates that both DFEs, that of novel mutations as well as that of segregating polymorphisms, are at equilibrium skewed towards the negative selection regime. 

\subsection*{Acknowledgements}

The authors thank Nicolas Lartillot and Thibault Latrille for valuable discussions about the use of mutation-selection models for protein-coding sequence evolution. CFM has received financial support from the Knut and Alice Wallenberg Foundation (2014/0044 to Hans Ellegren) and the Swedish Research Council (2013-8271 to Hans Ellegren).

\bibliography{manuscript.bib}

\end{document}